\documentclass[reqno]{amsart}
\usepackage{amsmath}
\usepackage{amsthm}
\usepackage{amssymb}
\usepackage{epsfig}
\usepackage{psfrag}
\usepackage{graphicx}
\usepackage{pgf}
\usepackage{mathtools,color}
\usepackage{esint,amsfonts,amsmath,amssymb,epsfig,mathrsfs,mathtools}

\newcommand{\Rz}{\mathbb{R}}
\newcommand{\Rzd}{\mathbb{R}^{d {\times} d}_{\rm dev}}
\newcommand{\Rzs}{\mathbb{R}^{d {\times} d}_{\rm sym}}
\newcommand{\Rzn}{\mathbb{R}^{d {\times} d}}
\newcommand{\sy}{{\rm sym}}
\newcommand{\an}{{\rm anti}}
\newcommand{\dev}{{\rm dev}}

\newcommand{\disp}{\displaystyle}
\newcommand{\haz}{\widehat}
\newcommand{\epsi}{\varepsilon}

\newcommand{\dx}{\,\text{\rm d}x}
\newcommand{\dt}{\,\text{\rm  d}t}
\newcommand{\ds}{\,\text{\rm  d}s}
\newcommand{\SL}{\text{\rm SL}(d)}
\newcommand{\SO}{\text{\rm SO}(d)}
\newcommand{\GLp}{\text{\rm GL}_+(d)}
\newcommand{\Wel}{W_{\rm el}}
\newcommand{\Wh}{W_{\rm h}}
\newcommand{\Whf}{\widetilde W_{\rm h}}
\newcommand{\I}{I}

\newcommand{\C}{{\mathbb C}}
\renewcommand{\H}{{\mathbb H}}
\renewcommand{\t}{{\top}}
\newcommand{\EE}{{\mathcal E}}
\newcommand{\DD}{{\mathcal D}}
\newcommand{\WW}{{\mathcal W}}
\renewcommand{\SS}{{\mathcal S}}
\newcommand{\QQ}{{\mathcal Q}}
\newcommand{\HH}{{\rm H}}
\newcommand{\LL}{{\rm L}}
\newcommand{\CC}{{\rm C}}
\newcommand{\UU}{{\mathcal U}}
\newcommand{\ZZ}{{\mathcal Z}}
\newcommand{\lan}{\langle}
\newcommand{\ran}{\rangle}
\newcommand{\dist}{{\rm dist}}
\newcommand{\Diss}{{\rm Diss}}

\def\Gliminf{\mathop{\Gamma\text{--}\mathrm{liminf}}}
\def\Glimsup{\mathop{\Gamma\text{--}\mathrm{limsup}}}
\def\Argmin{\mathop{\mathrm{Arg}\,\mathrm{min}}}
\newcommand{\id}{{\rm id}}

\newcommand{\el}{{\rm el}}
\newcommand{\cof}{{\rm cof}}
\newcommand{\ove}{\overline}

\newtheorem{theorem}{Theorem}[section]

 \newtheorem{corollary}[theorem]{Corollary}
\newtheorem{definition}[theorem]{Definition}
 
\newtheorem{lemma}[theorem]{Lemma}
\begin{document}

\title[Linearized plasticity as {$\Gamma$}-limit of finite
              plasticity]{Linearized plasticity is the evolutionary
               \\ {$\Gamma$}-limit 
               of finite
             plasticity}

\author{Alexander Mielke}
\address[Alexander Mielke]{Weierstra\ss-Institut f\"ur Angewandte Analysis und Stochastik, 
Mohrenstra\ss e 39, D-10117 Berlin and Institut f\"ur Mathematik, 
Humboldt-Universit\"at zu Berlin, Rudower Chaussee 25, D-12489 Berlin, Germany
}
\email{mielke\,@\,wias-berlin.de}
\urladdr{http://www.wias-berlin.de/people/mielke/}

\author{Ulisse Stefanelli}
\address[Ulisse Stefanelli]{Istituto di Matematica Applicata e Tecnologie Informatiche E. Magenes - CNR, 
v. Ferrata 1, I-27100 Pavia, Italy and Weierstra\ss-Institut f\"ur Angewandte Analysis und Stochastik, 
Mohrenstra\ss e 39, D-10117 Berlin, Germany}
\email{ulisse.stefanelli\,@\,imati.cnr.it} 
\urladdr{http://www.imati.cnr.it/ulisse/}
\thanks{A. Mielke was partially supported by DFG within
the Research Unit FOR 797 (subproject P5, Mie 459/5-2). U. Stefanelli was partially supported by FP7-IDEAS-ERC-StG Grant
  \#200497 {\it BioSMA}, the CNR-AV\v CR grant {\it SmartMath}, and the
  Alexander von Humboldt Foundation}

\begin{abstract} 
We provide a rigorous
justification of the classical linearization approach in
plasticity. By taking the small-deformations limit, we prove via
$\Gamma$-convergence for rate-independent processes that energetic solutions of
the quasi-static finite-strain elastoplasticity system converge to the unique strong solution of
linearized elastoplasticity.
\end{abstract}

\maketitle 

\pagestyle{myheadings}

                                %
                                %
                                %

\section{Introduction}\label{intro}
\setcounter{equation}{0}

This paper is devoted to the rigorous justification of
the classical linearization approach in finite-strain elastoplasticity. When
restricting to the
small-deformation realm it is indeed customary to leave the nonlinear
finite-strain frame and resort to 
linearized theories instead. This reduction is usually
motivated by means of heuristic Taylor expansion
arguments. Here, we aim at complement these formal motivations by
providing a rigorous linearization proof by means of an {\it evolutionary
$\Gamma$-convergence} analysis of rate-independent processes. In particular, we
address the general time-dependent case, which e.g. allows for cyclic loading.

In the stationary framework, the pioneering contribution in this context goes back to {\sc Dal
  Maso, Negri, \& Percivale} \cite{DalMaso-et-al02} who devised a
convergence proof of finite-strain elasticity to
linearized elasticity. Later, the argument  has been refined by
{\sc Agostiniani, Dal Maso, \& DeSimone} \cite{Agostiniani11} and  extended
to multi-well energies by {\sc Schmidt} \cite{Schmidt08} and to residually stressed materials by {\sc Paroni \& Tomassetti} \cite{Paroni-Tomassetti09,Paroni-Tomassetti11}. The reader
is also referred to
\cite{Gloria-Neukamm10,Mueller-Neukamm11,Neukamm10} for
some related results in the direction of homogenization, to
\cite{Agostiniani-DeSimone11} for an application to the study of
nematic elastomers, to \cite{Braides-Solci-Vitali07,Schmidt09} in the
context of convergence of atomistic models, and to
\cite{Scardia-Zeppieri11} in relation with dislocation theory.

To our knowledge, this is the first result in the {\it evolutionary}
case. With respect to the stationary case of \cite{DalMaso-et-al02},
the evolution situation is quite more involved.
Indeed, the argument in \cite{DalMaso-et-al02} relies on the
$\Gamma$-convergence proof of the small-deformation energy functional
to its linearization limit. Here, we are instead forced to cope with
the occurrence of dissipative plastic evolution by means of a delicate
recovery sequence construction relating energy and dissipation. We
emphasize that finite-strain elastoplasticity is based on the
multiplicative decomposition of the strain tensors. Moreover, the
plastic tensor is to be considered as an element of a multiplicative matrix group. We have
to control these noncommutative {\it multiplicative} structures in linear
function spaces and to establish their convergence to the
corresponding {\it linear} additive
structures.
In order to give some details in this
direction we cannot avoid introducing some minimal notation. 

Finite-strain elastoplasticity is usually based on the multiplicative
decomposition $\nabla \varphi= F_{\rm el} F_{\rm pl}$ \cite{Lee69}. Here
$\varphi:\Omega \to \Rz^d$ is the deformation of the body with respect
to the 
reference configuration $\Omega \subset \Rz^d$ $(d=2,3)$ while $F_{\rm el}$
 and $F_{\rm pl}\in\SL$ stand for the elastic and the plastic strain,
 respectively. Then, the stored energy in the body is
 written as
$$ \int_\Omega \Wel (\nabla \varphi F_{\rm pl}^{-1}) \dx +
\int_\Omega \Wh(F_{\rm pl}) \dx$$
where $\Wel$ is a frame-indifferent elastic stored-energy density and $\Wh$
describes hardening. The plastic flow rule is
expressed by means of a suitably defined dissipation distance $D: \SL \times \SL \to
[0,\infty]$. In particular $D(F_{\rm pl},\haz F_{\rm pl})$ represents
the minimal dissipated energy for an evolution from the plastic strain $F_{\rm
  pl}$ to $\haz F_{\rm pl}$ and is given via a positively
$1$-homogeneous dissipation function $R$ by
$$D(F_{\rm
  pl},\haz F_{\rm pl}) = D(\I, \haz F_{\rm
  pl}F_{\rm
  pl}^{-1} )= \inf \int_\Omega \int_0^1R(\dot P P^{-1}) \dt \dx,$$
the infimum being taken among all smooth trajectories $P:[0,1]\to \Rzn$
connecting $F_{\rm
  pl}$ to $\haz F_{\rm pl}$. Staring from these functionals, by
specifying loadings, boundary, and initial conditions, suitably weak solutions of the quasi-static
finite-plasticity system (see
Section \ref{result}) can
be defined. We refer to \cite{Mielke03} for more information on the
mathematical modeling of finite-strain elastoplasticity. There also
models with additional hardening variables are given. Here we however refrain
from maximal generality in order to emphasize the main
features of the limiting process.

Let now the deformation and the plastic strain be small. In
particular, for $\epsi>0$ let $\varphi_\epsi = \id {+} \epsi u$ and
$F_{\rm pl,\epsi}= \I {+} \epsi z$ 
 where $u$ is interpreted as the displacement and $z$ is the
linearized plastic strain. 
Correspondingly, we have that $F_{\rm
el,\epsi} =\nabla \varphi_\epsi F_{\rm pl,\epsi}^{-1}= (\id{+}\epsi \nabla
u)(\I{+}\epsi z)^{-1}$ and we are lead to the consideration of
the  rescaled  finite-strain elastoplasticity energy and dissipation functionals
\begin{align}
& \frac{1}{\epsi^2} \int_\Omega \Wel \big((\I{+}\epsi \nabla u)(\I {+}
 \epsi z)^{-1}\big) \dx +\frac{1}{\epsi^2} 
\int_\Omega \Wh(\I{+}\epsi z) \dx,\nonumber\\
&\frac{1}{\epsi}D((\I{+}\epsi z), (\I{+}\epsi \haz z)).\nonumber
\end{align}
Note that the rescalings above are such that, by assuming 
$\Wel$ and $\Wh$ to admit a quadratic expansion around identity, one
can check that 
\begin{align}
\frac{1}{\epsi^2}\int_\Omega\Wel \big((\I{+}\epsi \nabla u)(\I {+}
 \epsi z)^{-1}\big) \dx \ &\to \  \frac12 \int_\Omega (\nabla u{-}z){:} \C
 (\nabla u{-}z)\dx,\nonumber\\
\frac{1}{\epsi^2} \int_\Omega\Wh(\I{+}\epsi z) \dx\ &\to  \ \frac12
 \int_\Omega z{:} \H z \dx\nonumber\\
 \frac{1}{\epsi}D((\I{+}\epsi z), (\I{+}\epsi \haz z)) \ &\to \  \int_\Omega
 R(\haz z {-} z) \dx.\nonumber
\end{align}
This pointwise convergence is the classical
justification of linearization in plasticity. On the other hand, it is
not sufficient in itself for proving that finite-strain elastoplasticity
trajectories actually convergence to a solution of the 
linearized-plasticity system. 

 Before going on let us mention that the solution concept which is here
under consideration is that of {\it energetic solutions}. Starting
from \cite{Mielke-Theil04}, this solution notion has been extensively
applied in many different rate-independent contexts. We shall however
 record that one of the main motivations for introducing energetic
solutions was exactly that of targeting existence theories for finite-strain
elastoplasticity. In this respect, note that the only available existence result
for finite-strain elastoplastic evolution has been recently obtained within 
the {energetic solvability} frame in \cite{Mainik-Mielke09} after
adding the regularizing term $|\nabla F_{\rm pl}|^r$ for $r >1$ (see
also \cite{Mielke-Mueller06} for some preliminary result),

Our result consists in proving the convergence of energetic solutions
of the finite-strain elastoplasticity system to linearized-plasticity solutions.
In order to prove this
convergence we follow the abstract evolutionary $\Gamma$-convergence theory for
energetic solutions of rate-independent processes developed in
\cite{mrs}. We shall
mention that this evolutionary $\Gamma$-convergence method has recently
attracted attention and has been successfully considered in connection with numerical approximations
\cite{Kruzik05,Mielke-Roubicek09,Giacomini-Ponsiglione06bis}, damage
\cite{Bouchitte-et-al09,Thomas-Mielke10}, fracture \cite{Giacomini-Ponsiglione06},
delamination \cite{Roubicek-et-al09}, dimension reduction \cite{Freddi-et-al10,Liero-Mielke11},
homogenization \cite{Timofte09}, and optimal control \cite{Rindler08,Rindler09}.

According to \cite{mrs}, the convergence of the trajectories
$(u_\epsi,z_\epsi)$ follows by
proving two {\it separate} $\Gliminf$
inequalities for energy and dissipation and constructing of a {\it
  mutual recovery sequence} relating both. Note that separate
$\Gamma$-convergence for energy and dissipation is not sufficient to
pass to the limit within rate-independent processes. Apart from the additional technicalities due to the
presence of the plastic strain and the dissipation functional, it is the
delicate construction of the mutual recovery sequence that distinguishes
our argument from all the already developed stationary analyses in the
spirit of~\cite{DalMaso-et-al02}.


\section{Problem setup and results}\label{result}
\setcounter{equation}{0}

Let the reference configuration $\Omega\subset \Rz^d$ be an open,  bounded, and connected  set with
Lipschitz boundary. Moreover, let $\Gamma\subset \partial \Omega$ be relatively
open with ${\mathcal H}^{d-1}(\Gamma)>0$. We define the state space as
$$ \QQ:=\UU \times \ZZ := \big\{u \in \HH^1(\Omega;\Rz^d) \ | \ u=0 \ 
\text{on} \ \Gamma\big\} \times \LL^2(\Omega;\Rzn).$$
Note that the choice of the homogeneous Dirichlet condition on the displacement $u$
is just motivated by the sake of simplicity. In particular, different
boundary conditions may be considered as well.

For all given $A \in \Rzn$ we denote its symmetric and
antisymmetric parts as
$A^\sy := (A {+} A^\t)/2$ and $A^\an =
A {-}A^\sy$.
We indicate by $\Rzs$  and $\Rzn_{\an}$ the subspaces of symmetric and
antisymmetric tensors, respectively,  whereas $\Rzd$
stands for the subspace of symmetric and trace-free tensors, also
called {\it deviatoric} tensors.
The standard Euclidian tensor norm is denoted by $|\cdot|$ and, for
all $A\in \Rzn$ and $\tau>0$,
$B_\tau(A)$ indicates the ball 
$B_\tau(A):=\{B \in \Rzn \ | \ |A-B| < \tau\}$. Moreover, the symbol
$|\cdot|_{\mathbb T}$  stands for the seminorm 
$$|A|_{\mathbb T}^2 := \frac12 A {:}{\mathbb T}A$$
where the $4$-tensor ${\mathbb T}\in\Rz^{d\times d \times d \times d}$
is symmetric (${\mathbb T}_{ij k \ell} = {\mathbb T}_{k \ell i j}$) and positive
semidefinite. For finite-strain elastoplasticity we use the classical notations
\begin{align}
  \SL&:=\{P \in \Rzn \ | \ \det P =1\},\nonumber\\
\SO&:=\{R \in  \SL  \ | \ R^\t R = R R^\t = \I \},\nonumber\\
\GLp&:=\{Q \in \Rzn \ | \ \det Q >0 \}.\nonumber
\end{align}

We assume that the elastic energy density functional $\Wel$ 
fulfills
\begin{subequations}
\begin{align}
&   \Wel: \Rzn \to [0,\infty], \ \ \Wel \in {\rm C}^1(\GLp ), \ \ \Wel\equiv \infty \ \text{on} \ \Rzn\setminus \GLp, \label{smoothness}\\
&\forall F \in \GLp \  \forall R \in \SO:\
\Wel(RF) = \Wel(F),\label{frame}\\
 & \forall F \in \GLp: \ \Wel(F) \geq c_1  \dist^2(F,\SO),\label{distSO}\\
 & \forall F \in \GLp: \  |F^\t \partial_F \Wel(F)  | \leq c_2(\Wel(F) +
c_3), \label{Mandel}\\
&\exists \C \geq 0\ \forall \delta>0 \ \exists c_{\rm el}(\delta)>0\ \forall A \in
B_{c_{\rm el}(\delta)}(0):  \ \big|\Wel(\I {+} A) - |A|_\C^2 \big| \leq \delta |A|_\C^2,\label{approx}
\end{align}\label{energy_ass}
\end{subequations}
\!\!\! for some positive $c_1, \, c_2$.
Assumption \eqref{frame} is nothing but frame indifference and the
nondegeneracy requirement \eqref{distSO} is quite classical. Assumption \eqref{Mandel} entails the controllability of the {\it Mandel
tensor} $F^\t \partial_F \Wel(F) $ by means of the energy. This is a
crucial condition in finite-strain elastoplasticity
(cf. \cite{Ball84,Ball02}) and was used in the context of rate-independent
processes in 
\cite{Francfort-Mielke06,Mainik-Mielke09}. Condition,
\eqref{approx} encodes the local quadratic character of $\Wel$ around
identity. More precisely, \eqref{approx} states that $|\cdot|_\C$
is the second order Taylor expansion of $\Wel$ at $\I$, and may be reformulated by
saying that $A \mapsto \Wel(\I{+}A)$ is locally restrained between two multiples of
$|\cdot|_\C^2$, namely, 
$$\forall \delta>0 \ \forall A \in B_{c_{\rm el}(\delta)}(0) : \quad  (1{-}\delta) |A|_\C^2 \leq
\Wel(\I{+}A) \leq  (1{+}\delta)|A|^2_\C.$$
Moreover, \eqref{approx} entails
\begin{equation}\label{stress-free}
\Wel(\I)=0, \ \ \partial_F \Wel(\I)=0, \ \ \partial_F^2\Wel(\I) = \C,
\end{equation}
which, in particular, yields that the reference state is
stress free. On the other hand, by assuming \eqref{stress-free} and letting
$\Wel\in\CC^2$ in neighborhood of $\I$, relation \eqref{approx}
follows.

 Note that the symmetry of the elastic tensor $\C$ (implicitly
assumed in the notation $|\cdot|_\C$) may be directly obtained
from the last of \eqref{stress-free} by assuming additional smoothness
on $\Wel$. Moreover, letting $A\in \Rzn$ be given, as we have that
$\exp(A^\an) \in \SO$, the frame indifference \eqref{frame}
entails that the function $t \mapsto \partial_F \Wel(\exp(tA^\an))$ is
constantly equal to $\partial_F \Wel(\I)=0$. Hence, by taking 
its derivative with respect to $t$ and evaluating it at $t=0$ we get
$ \C A^\an=0$.
Namely, $\C$ necessarily fulfills also the so called {\it minor
  symmetries} $\C_{ij  k \ell} = \C_{j i  k \ell} = \C_{ij \ell k}$
and we have 
 \begin{equation}\label{minor}\forall A \in \Rzn: \ \C A= \C A^\sy.
 \end{equation}

On the other hand, as effect of the nondegeneracy \eqref{distSO} and assumption
\eqref{approx} we have that $\C$ is positive definite on $\Rzs$. Indeed, by linearizing $d(\cdot,\SO)$ around identity
we have \cite[{(3.21)}]{Friesecke-et-al02}
\begin{equation}
\forall B \in \Rzn: \quad d(B,\SO) = |B^\sy {-} \I| + O(|B {-} \I|^2).\label{fri}
\end{equation}
Hence, given $A \in \Rzn$ and $\eta, \, \delta>0$, by choosing
$B=\I{+}\eta A$ in the latter we have
\begin{align}
c_1 |A^\sy|^2 &\stackrel{\eqref{fri}}{=} \lim_{\eta \to 0}
\frac{c_1}{\eta^2} d^2(I{+}\eta A,\SO)\nonumber\\
&\stackrel{\eqref{distSO}}{\leq} \lim_{\eta \to 0}
\frac{1}{\eta^2}\Wel(I{+}\eta A) \stackrel{\eqref{approx}}{\leq}
(1{+}\delta) | A|^2_\C  
\nonumber
\end{align}
so that, by taking $\delta \to 0$, we have
\begin{equation} \forall A \in \Rzn: \quad c_1|A^\sy|^2 \leq  |
  A|^2_\C = |
  A^\sy|^2_\C.\label{posdef}
\end{equation}

Note that all assumptions \eqref{smoothness}-\eqref{approx} are consistent with the usual polyconvexity
framework
\begin{align}
&F \mapsto \Wel(F) \ \  \text{polyconvex,} \nonumber\\
& \Wel(F) \to  \infty \ \  \text{for}    \ \ \det F \to
0.\nonumber
\end{align}


Our assumptions on the hardening functional $\Wh: \Rzn \to [0,\infty]$ read
\begin{subequations}\label{energyh_ass}
  \begin{align}
    & \Wh(P):= \left\{
      \begin{array}{ll}
        \Whf(P) \quad &\text{if} \ \ P \in K,\\
        \infty &\text{if} \ \ P \in \Rzn\setminus K ,
      \end{array}
    \right.\label{quad}\\
    & \text{where} \ \ K \ \ \text{is compact in $\SL$ and contains a
      neighborhood of
      $\I$,}\label{compdom}\\
    & \Whf:\Rzn \to \Rz \ \ \text{is locally Lipschitz continuous and}\label{lip}\\
    & \exists \H \geq 0\ \forall \delta>0 \ \exists c_{\rm
      h}(\delta)>0\ \forall A \in B_{c_{\rm h}(\delta)}(0): \
    \big|\Whf(\I {+} A) - |A|_\H^2 \big| \leq
    \delta |A|_\H^2,\label{approx2}\\
    & \exists c_3>0 \ \forall A \in \Rzn: \ \Wh(\I {+} A) \geq c_3
    |A|^2.\label{coercwh}
  \end{align}
\end{subequations}
Note that by
assumption \eqref{compdom} we can find a constant $c_K>0$ such that 
\begin{align}
  &P \in K \ \ \Rightarrow \ \ |P| + |P^{-1}| \leq c_K, \label{conscomp}
  \\ &P \in \SL \setminus K \ \ \Rightarrow \ |P - \I|\geq \frac{1}{c_K}.\label{conscomp2}
\end{align}
The rather strong technical assumption on $\Wh$ that
its effective domain $K=\{P \in \SL \ | \ \Wh(P) < \infty\}$ fulfills
\eqref{conscomp} is crucial as it will provide $\LL^\infty$-bounds
that are essential in order to control the multiplicative terms $(\I{+}\epsi
\nabla u)(\I{+}\epsi z)^{-1}$.  Moreover, by combining
\eqref{approx2} and \eqref{coercwh} we check that 
\begin{equation}
  \label{posdef2}
  \forall A \in \Rzn: \quad  c_3|A|^2 \leq  | A|^2_\H.
\end{equation}

As for the dissipation we assume that
\begin{subequations}
  \begin{align}
    &  R^\dev: \Rzd \to  [0,\infty)  \ \ \text{convex and positively $1$-homogeneous,}\label{R1}\\
    &\forall P \in \Rzd: \ \quad c_4|P|\leq R^\dev(P)\leq c_5|P|,\label{R2}\\
    & R: \Rzn \to [0,\infty]; \ \ R(z):= \left\{
      \begin{array}{ll}
        R^\dev(z)\quad& \text{if}\ \ z \in \Rzd,\\
        \infty & \text{else,}
      \end{array}
    \right.\label{R3}
  \end{align}\label{diss_ass}
\end{subequations}
for positive $c_4,\, c_5$. Moreover, we define
\begin{align}
&  D: \Rzn \times \Rzn \to [0,\infty], \ \ \text{with} \ \  D(P,\haz P) =
D(\I,\haz P P^{-1}) \ \ \text{given by}\nonumber\\
&  D(\I,\haz P):= \inf \Bigg\{\int_0^1 R(\dot P P^{-1}) \dt \ \Big|
  \nonumber\\
&  \hspace{30mm}P\in \CC^1(0,1;\Rzn), \ P(0)=\I, \ P(1)=\haz P\Bigg\} .\label{Ddef}
\end{align}
 If $P$ is not invertible, we set $D(P,\haz P) = \infty$. 
Note in particular that $D(I,P)<\infty$ implies $\det P =1$. Moreover,
 there exists $c_6>0$ such that 
\begin{equation}
  \label{Dbound}
\forall P,\,Q \in K\subset \SL: \  D(P, Q ) \leq c_6, \quad D(\I, P)\leq c_6|P{-}\I|.
\end{equation}
For the first estimate the continuity of $D$ and the compactness of
$K$ is sufficient. For the second, we need to establish the estimate
only for $P$ close to $I$, where it follows from $D(I,P) \leq
R^\dev(\log P)\leq c_5|\log P|\leq c_6|P{-}I|$, since the matrix
logarithm is well-defined and Lipschitz continuous in a neighborhood
of $I$. See also \cite[Ex.\,3.2]{Mainik-Mielke09} and the references
given there for global bounds on $D$. 

The quasistatic evolution of the finite-strain and linearized
elastoplasticity systems are driven by the energy
functionals $\WW_\epsi, \, \WW_0: \QQ \to (-\infty,\infty]$ given by
\begin{align}
&\WW_\epsi(u,z):= \frac{1}{\epsi^2}\int_\Omega \Wel \big((\I{+}\epsi \nabla
  u)(\I {+} \epsi z)^{-1}\big) \dx + \frac{1}{\epsi^2}\int_\Omega \Wh(\I {+}
  \epsi z) \dx, \nonumber\\
&\WW_0(u,z):= \int_\Omega |\nabla u^\sy {-} z^\sy|_\C^2 
\dx + \int_\Omega|z|_\H^2 \dx.\nonumber
\end{align}
 Note that, if the second integral in the definition of
$\WW_\epsi(u,z)$ is finite, then $\I{+} \epsi z \in K$ almost
everywhere by \eqref{quad}. Hence, the inverse $(\I{+}\epsi z)^{-1}$
exists and the first integral is well defined.

We prescribe the generalized loading as
\begin{equation}\label{ell}
 \ell \in W^{1,1}(0,T;\UU')
\end{equation}
and, by letting $\ell_\epsi := \epsi \ell$, we introduce some
notation for the total energy functionals
 $\EE_\epsi, \, \EE_0: [0,T] \times \QQ \to (-\infty,\infty]$ 
 as
\begin{align}
&\EE_\epsi(t,u,z):= \WW_\epsi(u,z) - \frac{1}{\epsi}\lan \ell_\epsi(t), u \ran =\WW_\epsi(u,z) - \lan \ell(t), u \ran 
\nonumber\\
&\EE_0(t,u,z):=\WW_0(u,z)   - \lan \ell(t), u
\ran,\nonumber
\end{align}
Eventually, the dissipative character of the evolution is encoded into
the dissipation functions $D_\epsi, D_0: \Rzn\times \Rzn \to [0,\infty]$ and functionals $\DD_\epsi, \DD_0: (L^1(\Omega;\Rzn))^2
\to [0,\infty]$ given by
\begin{align}
 &D_\epsi(z_1,z_2):= \frac{1}{\epsi}D(\I{+}\epsi z_1, \I {+}\epsi
 z_2), \quad D_0(z_1,z_2):= R(z_2 {-} z_1),\nonumber\\
&\DD_\epsi(z_1,z_2):= \int_\Omega D_\epsi(z_1,z_2)  \dx, \quad
\DD_0(z_1,z_2):= \int_\Omega D_0(z_1,z_2)  \dx. \nonumber
\end{align}
The total dissipation of the process over the time interval
$[0,t]\subset [0,T]$
will be given by 
$$\Diss_{\DD_\epsi}(z;[0,t]) := \sup\left\{ \sum_{i=1}^N
  \DD_\epsi(z(t^i) , z(t^{i-1})) \ | \ \{0=t^0< \dots< t^N=t\}\right\}$$
where the sup is taken over all partitions of $[0,t]$.

From here on, we term   {\it
  Rate-Independent System} (RIS) the triple $({\mathcal
  Q},\EE_\epsi,\DD_\epsi)$ given by the choice of the state space ${\mathcal Q}$ and the energy and dissipation
functionals $\EE_\epsi$ and $\DD_\epsi$. The term {\it
  evolutionary $\Gamma$-convergence} refers to a suitable notion of
convergence for rate-independent systems in the spirit of \cite{mrs}
which in particular entails the convergence of the
respective energetic solutions. 

A crucial structure in the energetic formulation of RIS is the set $\SS_\epsi(t)$ of {\it stable
  states} at time $t \in [0,T]$, which is defined via
\begin{align}
 \SS_\epsi(t):=\Big\{(u,z)\in \QQ \ |&
    \ \EE_\epsi(t,u,z)
    < \infty  \ \ \text{and}\nonumber \\
    &\EE_\epsi(t,u,z) \leq \EE_\epsi(t,\haz u,\haz z) +
    \DD_\epsi( z ,\hat z) \quad  \forall (\haz u, \haz z)
    \in \QQ   \Big\}.\nonumber
\end{align}

Our assumption on the initial data reads
\begin{align}
 &\SS_\epsi(0)\ni (u^0_\epsi,z^0_\epsi) \to (u^0_0,z^0_0)  \ \ \text{weakly in} \ \
  \QQ, \ z_0^0 \in  \LL^2(\Omega; \Rzd),  \nonumber\\
 & \EE_\epsi(0,u^0_\epsi,z^0_\epsi)\to \EE_0(0,u^0_0,z^0_0).   \label{initial_data}
\end{align}
Note that the latter assumption is not empty as it is fulfilled at
least by the natural choice $(u_0,z_0)=(0,0)$ if $\ell(0)=0$.

\begin{definition}[Energetic solutions] Let $\epsi \geq 0$. We say
  that a trajectory $q_\epsi : [0,T] \to (u_\epsi,z_\epsi) \in \QQ$ is an
  \emph{energetic solution (related to the ${\rm RIS}$ $(\QQ,\EE_\epsi,\DD_\epsi)$)} if $(u_\epsi(0),z_\epsi(0)) =
  (u^0_\epsi,z^0_\epsi)$, the map $t
  \mapsto \lan\dot \ell, u_\epsi \ran$ is integrable, and, for all $t
  \in [0,T]$,
  \begin{align}
    & (u_\epsi(t),z_\epsi(t))\in \SS_\epsi(t),\label{stability}\\
    &\EE_\epsi(t,u_\epsi(t),z_\epsi(t)) + {\rm
      Diss}_{\DD_\epsi}(z_\epsi;[0,t]) =
    \EE_\epsi(0,u^0_\epsi,z^0_\epsi) - \int_0^t \lan \dot \ell,
    u_\epsi \ran \ds.\label{energy}
  \end{align}
An energetic solution will be called a \emph{finite-plasticity
  solution} if $\epsi>0$ and a \emph{linearized-plasticity
  solution} for $\epsi=0$.
\end{definition}

Note that linearized-plasticity solutions $(u_0,z_0)$ are unique as effect of the
quadratic and uniformly convex character of $\WW_0$. Moreover, from
assumption \eqref{ell} we get that 
$(u_0,z_0)\in {\rm W}^{1,1}(0,T;\QQ)$ and
$$\forall t \in [0,T]: \ \Diss_{\DD_0}(z_0;[0,t]) = \int_0^t R(\dot
z_0)\ds.$$
The reader is referred to \cite{Hill50,Lubliner90,Martin75} for some general introduction to plasticity
and to \cite{Han-Reddy,Johnson76,Suquet81} for the classical well-posedness theory for linearized elastoplasticity.

Our main result reads as follows and will be proved in Section
\ref{proof} as a special instance of the general theory of \cite{mrs}.

\begin{theorem}[Finite plasticity $\Gamma$-converges to linearized
  plasticity]
Assume \eqref{energy_ass}-\eqref{energyh_ass}, \eqref{diss_ass}, and
\eqref{ell}-\eqref{initial_data}. Let $(u_\epsi,z_\epsi)$ be a finite-plasticity
solution. Then, $(u_\epsi(t),z_\epsi(t)) \to (u_0(t),z_0(t))$ weakly
in $\QQ$ for all $t \in [0,T]$ where $(u_0,z_0)$
is the unique linearized-plasticity solution.
\label{main}
\end{theorem}

Theorem \ref{main} is exclusively a convergence result. In particular,
we {\it assume} that finite-plasticity solutions
exist. Note however that the existence of finite-plasticity solutions is presently not known within our minimal assumption
frame. A possibility here would be that of considering directly some more
regular situations including extra
compactifying terms like $|\nabla F_{\rm pl}|^r \, (r>1)$ such that finite-plasticity solutions exist \cite{Mainik-Mielke09}. We shall not follow this line here but rather
present a second result based on {\it approximate
minimizers} of the related incremental problems. Indeed, given the time
partitions $\{0=t^i_\epsi< \dots< t^{N_\epsi}_\epsi=T\}$ with
diameters $\tau_\epsi:=\max_{i=1, \dots, N_\epsi} (t^i_\epsi - t^{i-1}_\epsi) \to
0$ as $\epsi \to 0$, the (iterative) incremental problem 
$$(u^i_\epsi,z^i_\epsi)\in \Argmin_{(u,v)\in \QQ}\big( \EE_\epsi(t^i_\epsi,u,z)
    + \DD_\epsi(z^{i-1}_\epsi,z)\big)\quad \text{for} \ \ i=1, \dots,N_\epsi$$
may not be solvable (cf. \cite{Carstensen-et-al02}, still see
\cite{Mielke04,Mielke-Mueller06} for some additional discussion). Hence, following \cite[Sec. 4]{mrs} we fix a sequence $0<\alpha_\epsi \to 0$ in
order to control the 
tolerances for the minimizations and consider the
following {\it approximate incremental problem}
\begin{equation}
  \label{AIP}
  \begin{array}{l}
  \text{Find iteratively} \ \ (u^i_\epsi,z^i_\epsi)\in \QQ \ \ \text{such
    that} \\ \EE_\epsi(t^i_\epsi,u^i_\epsi,z^i_\epsi)
    + \DD_\epsi(z^{i-1}_\epsi,z^i_\epsi) \\
\leq (t^i_\epsi - t^{i-1}_\epsi)\alpha_\epsi +
    \inf_{(u,v)\in \QQ} \big( \EE_\epsi(t^i_\epsi,u,z)
    + \DD_\epsi(z^{i-1}_\epsi,z)\big).
  \end{array}
\end{equation}
By the definition of infimum the latter always admits solutions and we
will show the following convergence result.

\begin{theorem}[Convergence of approximate incremental
  minimizers] Under the assumptions of Theorem \emph{\ref{main}} let $(u^i_\epsi,z^i_\epsi)$ be
approximate incremental minimizers and  $(\ove u_\epsi,\ove
z_\epsi)$ be the corresponding right-continuous, piecewise-constant interpolants on
the time partitions. Then,  $(\ove u_\epsi (t),\ove z_\epsi(t))
\to (u_0(t),z_0(t))$ weakly in $\QQ$ for all $t \in [0,T]$   where $(u_0,z_0)$
is the unique linearized-plasticity solution. \label{main2} 
\end{theorem}

In the finite-elasticity case (stationary), using ideas from \cite{DalMaso-et-al02} the convergence of
approximate minimizers has been considered in \cite{Paroni-Tomassetti09}.


\section{Proofs}\label{proof}
\setcounter{equation}{0}

The argument basically follows the lines of the abstract analysis of
\cite{mrs}. Still, our setting cannot be completely recovered from the
application of the above-mentioned abstract theory as extra care is
needed for the treatment of the multiplicative nonlinearities. We hence resort in providing
here an independent proof. After establishing the coercivity of
the energy in Subsection \ref{sec:coerc}, the proof strategy relies in
providing {\it two separate $\Gliminf$} inequalities for $\EE_\epsi$ and
$\DD_\epsi$ and a {\it mutual recovery sequence} argument relating
both. This is done in Subsections \ref{sec:Gliminf} and
\ref{sec:mutual} below. Eventually, the proofs of Theorems \ref{main}
and \ref{main2} are
outlined in Subsections \ref{sec:proof} and  \ref{sec:proof2}, respectively.

A caveat on notation: henceforth the symbol $c$
stands for any positive constant  independent of $\epsi$ and $\delta$ but possibly
depending on the fixed data. In particular, note that $c$ may change from line
to line. Moreover, in the following we use the short-hand
notation, for all $A\in \Rzn$,
\begin{align}
& \Wel^\epsi(A):= \frac{1}{\epsi^2}\Wel(\I {+} \epsi A), \quad
\Wh^\epsi(A):= \frac{1}{\epsi^2}\Wh(\I {+} \epsi A), \quad
\Whf^\epsi(A):= \frac{1}{\epsi^2}\Whf(\I {+} \epsi A).\nonumber
\end{align}

\subsection{Energy coercivity}\label{sec:coerc}

We start by providing a uniform coercivity result for the energy. It
follows the ideas in \cite{DalMaso-et-al02} and relies on the Rigidity
Lemma \cite[Thm. 3.1]{Friesecke-et-al02}.

\begin{lemma}[Coercivity] There exists $c>0$ such that, for all $(u,z) \in \QQ$\label{encoerc} 
  \begin{equation}
 \|\nabla u \|^2_{\LL^2} +  \| z\|^2_{\LL^2} + \| \epsi
 z\|^2_{\LL^\infty} \leq  c\big(1{+} \WW_\epsi(u,z)\big). \label{bound}
  \end{equation}
\end{lemma}

\begin{proof}
Let us assume with no loss of generality that
$\WW_\epsi(u,z)<\infty$,  so that $\I{+} \epsi z\in K$ almost
everywhere by assumption \eqref{quad}.  Hence, $|\I {+} \epsi z| \leq c_K$ almost everywhere
from property \eqref{conscomp}  and the inverse $(\I{+}\epsi
z)^{-1}$ exists almost everywhere.  Thus, we have that  $\|\epsi z\|_{\LL^\infty} \leq c$. Moreover, one
readily checks from the coercivity \eqref{coercwh} that 
\begin{equation}\label{tostack}c_3\| z\|_{\LL^2}^2 \leq \int_\Omega \Wh^\epsi(z) \dx\leq
\WW_\epsi(u,z).
\end{equation}

For the displacement $u$ we follow ideas from \cite{DalMaso-et-al02}. Given any $Q
\in \SO$ by letting $\varphi=\id{+}\epsi u$ and $F_{\rm el}= \nabla
\varphi (\I{+}\epsi z)^{-1}$ we have 
\begin{align}
  |\nabla \varphi {-} Q|^2 &=  |\nabla \varphi - Q(\I {+}\epsi z) + \epsi Q
  z|^2 = |(F_\el{-} Q) (\I{+}\epsi z) + \epsi Qz|^2 \nonumber\\
&\leq c\big( |F_\el{-} Q|^2|\I{+}\epsi z|^2 + \epsi^2 |z|^2\big) \leq c
\big(   |F_\el{-} Q|^2 + \epsi^2 |z|^2\big).\nonumber
\end{align}
In particular, by passing to the infimum for $Q \in \SO$ we have checked that
$$ \dist^2(\nabla \varphi,\SO)\leq c
\big(   \dist^2(F_\el,\SO) + \epsi^2|z|^2 \big).$$
By taking the integral in space and using the nondegeneracy condition \eqref{distSO}
we obtain that 
\begin{align}
\int_\Omega\dist^2(\nabla \varphi,\SO) \dx&\leq c
\int_\Omega  \dist^2(F_\el,\SO) \dx + c\epsi^2\int_\Omega |z|^2
\dx\nonumber\\
&\stackrel{\eqref{tostack}}{\leq} \epsi^2 c  
\WW_\epsi(u,z).\nonumber
\end{align}
Hence, the Rigidity Lemma \cite[Thm. 3.1]{Friesecke-et-al02} ensures that
$$\| \nabla\varphi {-} \haz Q\|_{\LL^2}^2 \leq
 \epsi^2 c  
\WW_\epsi(u,z)$$
for some constant rotation $\haz Q \in \SO$. Finally, using
\cite[Prop. 3.4]{DalMaso-et-al02} and $\varphi|_\Gamma = \id$ as $u
\in \UU$, we conclude $|\haz Q {-}
\I |^2 \leq
\epsi^2 c  \WW_\epsi(u,z)$. Then, we have 
$$\| \nabla u \|_{\LL^2}^2 = \frac{1}{\epsi^2} \| \nabla \varphi {-} \I
\|_{\LL^2}^2\leq \frac{2}{\epsi^2}\| \nabla \varphi {-} \haz Q
\|_{\LL^2}^2 +  \frac{2}{\epsi^2}\| \haz Q {-} \I
\|_{\LL^2}^2 \leq c  \WW_\epsi(u,z)$$
and the bound \eqref{bound} follows.
\end{proof}

\subsection{$\boldsymbol \Gliminf$ inequalities}\label{sec:Gliminf}

Next, we turn our attention to the proof of the separate $\Gliminf$
inequalities for energy and dissipation. Let us start with a
statement concerning the energy densities.

\begin{lemma} Under assumptions \eqref{approx} and \eqref{approx2}, we
  have 
  \begin{equation}\label{uni}\Wel^\epsi \to |\cdot|_\C^2 \ \
    \text{and} \ \ \Whf^\epsi\to|\cdot|_\H^2 \ \ \text{locally uniformly}. 
\end{equation}
Moreover, we have 
\begin{equation}\label{inq}|z|^2_\H \leq \inf\Big\{\liminf_{\epsi \to 0 } \Wh^\epsi(z_\epsi) \ \big| \
  z_\epsi \to z\Big\}.
\end{equation}\label{primo}
\end{lemma}

\begin{proof}
  Let $K_0 \Subset  \Rzn$, fix $\delta>0$ and find the
  corresponding $c_{\rm el}(\delta)>0$ from condition
  \eqref{approx}. As $\epsi K_0 \subset B_{c_{\rm
      el}(\delta)}(0)$ for $\epsi$ sufficiently small we
  have that 
$$ \limsup_{\epsi \to 0}\,\sup_{K_0}\big|\Wel^\epsi -
|\cdot|_\C^2\big|\leq \delta \sup_{K_0}|\cdot|^2 \leq \delta c$$
and local uniform convergence follows from $\delta>0$ being
arbitrary. The same argument applies to $\Whf^\epsi$.

As for the $\Gliminf$ inequality \eqref{inq}, let $z_\epsi \to z$ and assume with no loss of
generality that $\sup_\epsi\Wh^\epsi(z_\epsi) <\infty $. Hence,  $\Wh^\epsi(z_\epsi)=
\Whf^\epsi(z_\epsi)$ and the inequality follows from the above proved
uniform convergence.
\end{proof}

We are now in the position of proving the
$\Gliminf$ estimate for the energy. It follows indeed from \eqref{uni} and
the lower-semicontinuity result of Lemma \ref{balder}.

\begin{lemma}[$\Gliminf$ for the energy] \label{Genergy} 
For all $ (u,z) \in \QQ$ we have
$$\WW_0(u,z) \leq \inf\Big\{ \liminf_{\epsi
    \to 0} \WW_\epsi(u_\epsi,z_\epsi) \ \big| \ (u_\epsi,z_\epsi) \to
   (u,z)  \ \text{weakly in} \ \QQ\Big\}.$$
\end{lemma}
\begin{proof}
  Let $(u_{\epsi},z_{\epsi}) \to (u,z)$ weakly in $\QQ$. We can assume with no loss of
  generality that $\sup_\epsi \WW_{\epsi} (u_{\epsi},z_{\epsi}) < 
  \infty$. Owing to the $\Gliminf$ inequality
  \eqref{inq} and
the lower semicontinuity Lemma \ref{balder} we readily conclude that  
\begin{equation}\int_\Omega |z|^2_\H \leq \liminf_{\epsi\to 0} \int_\Omega
 \Wh^\epsi(z_\epsi)  \dx= \liminf_{\epsi\to 0} \frac{1}{\epsi^2}\int_\Omega
 \Wh(\I {+} \epsi z_\epsi)\dx.\label{ineqh}
\end{equation}

Moreover, $\WW_\epsi(u_\epsi,z_\epsi)<\infty$ implies $\epsi z_{\epsi}\in K{-}\I$ almost
  everywhere. In particular, $\epsi z_\epsi$ are
   bounded in $\LL^\infty$. The same holds for $(\I {+} {\epsi}
  z_{\epsi})^{-1}$ as
$$(\I {+} {\epsi}
  z_{\epsi})^{-1} =  \cof (\I {+} {\epsi} z_{\epsi})  /\det (\I {+} {\epsi}  z_{\epsi})  = \cof (\I {+} {\epsi}
  z_{\epsi}).$$
We define the auxiliary tensors 
\begin{equation}\label{we}
w_{\epsi} := \frac{1}{{\epsi}}\Big( (\I {+} {\epsi} z_{\epsi})^{-1} - \I +
{\epsi} z_{\epsi} \Big)= \epsi (\I {+} \epsi z_\epsi)^{-1} z_\epsi^2,
\end{equation}
so that $(\I {+} {\epsi} z_{\epsi})^{-1} = \I {-} {\epsi} z_{\epsi} {+} {\epsi}
w_{\epsi}$. By the first equality in \eqref{we} we have $\| {\epsi} w_{\epsi}
\|_{\LL^\infty} \leq c$, while the second gives
\begin{align}
 \| w_{\epsi}\|_{\LL^1}  
  = {\epsi} \| (\I {+} {\epsi}
   z_{\epsi})^{-1} z_{\epsi}^2\|_{\LL^1} \leq c{\epsi} \| z_\epsi
   \|^2_{\LL^2} \leq c\epsi\nonumber
\end{align}
where we have also used the boundedness in $\LL^2$ of $z_\epsi$
from \eqref{bound}.
Thus, by interpolation, $w_{\epsi}$ is  bounded
in $\LL^2$ as well, so that $w_\epsi \to 0$ weakly in $\LL^2$. 

Given $A_\epsi := (F_{\rm el,\epsi} {-}\I)/\epsi$ we want to show the
weak $\LL^2$
convergence $A_\epsi\to\nabla u{-} z$. From
\begin{equation}
A_\epsi=\frac{1}{\epsi}\left(  (\I
 {+} {\epsi} \nabla u_{\epsi})(\I {+} {\epsi} z_{\epsi})^{-1} -
 \I\right)\label{Aepsi}
\end{equation}
we find $\I{+}\epsi A_\epsi = (\I{+}\epsi \nabla
u_\epsi)(\I {+} \epsi z_\epsi)^{-1}$ and compute that 
\begin{align}
  A_\epsi=\frac{1}{\epsi}\big(  (\I
 {+} {\epsi} \nabla u_{\epsi})(\I {-} {\epsi} z_{\epsi}{+}\epsi w_\epsi) -
 \I\big) = \nabla u_\epsi - z_\epsi + w_\epsi - \epsi\big(\nabla
u_\epsi z_\epsi {-} \nabla u_\epsi w_\epsi\big).\nonumber
\end{align}
Hence, as we have that $\nabla u_\epsi {-} z_\epsi \to \nabla u {-} z$ and
$w_\epsi \to 0$ weakly in
$\LL^2$, we have to show $v_\epsi := \nabla u_\epsi (\epsi z_\epsi{-} \epsi
w_\epsi)\to 0$  weakly in $\LL^2$ as well. Indeed, the boundedness in $\LL^2$ 
of $v_\epsi$ follows from $\|
\nabla u_\epsi \|_{\LL^2}\leq c$ (see \eqref{bound}) and the
$\LL^\infty$-boundedness of $\epsi z_\epsi$ and $\epsi
w_\epsi$. Moreover, since
$z_\epsi$ and $w_\epsi$ are bounded in $\LL^2$ we have $\| v_\epsi
\|_{\LL^1} \leq c\epsi$ and conclude $v_\epsi \to 0$  weakly in $\LL^2$.

Eventually, owing to Lemma \ref{primo}, we are in the
position of exploiting the lower semicontinuity Lemma
\ref{balder} in order to obtain that 
\begin{align}
 \int_\Omega |\nabla u {-} z|^2_\C &\leq \liminf_{\epsi\to 0} \int_\Omega
 \Wel^\epsi(A_\epsi)  \dx\nonumber\\
&= \liminf_{\epsi\to 0} \frac{1}{\epsi^2}\int_\Omega
 \Wel\big((\I{+} \epsi \nabla u_\epsi)(\I {+} \epsi z_\epsi)^{-1}\big)\dx.\nonumber
\end{align}
Finally, by recalling relation \eqref{minor} and the already
established \eqref{ineqh} the assertion follows.
\end{proof}

Before moving to the $\Gliminf$ inequality for the
dissipation functionals $\DD_\epsi$, we prepare here a preliminary
result on the functions $D_\epsi$.

\begin{lemma}[$\Gamma$-convergence of $D_\epsi$]\label{GD} 
  $D_\epsi \to D_0$ in the sense of $\Gamma$-convergence.
\end{lemma}

\begin{proof}
{\it $\Gliminf$ inequality.}
Let $(z_\epsi,\haz z_\epsi) \to (z,\haz z)$ and assume with no loss of
generality that $ \sup_{\epsi }D_\epsi(z_\epsi, \haz z_\epsi)
<\infty$. In particular, we have that $ (\I{+}\epsi \haz
z_\epsi)(\I{+}\epsi  z_\epsi)^{-1} \in \SL$. By defining 
$$\zeta_\epsi := \frac{1}{\epsi}\big( (\I{+}\epsi \haz
z_\epsi)(\I{+}\epsi  z_\epsi)^{-1} -\I\big) = \haz z_\epsi - z_\epsi
+w_\epsi 
- \epsi \haz z_\epsi z_\epsi +  \epsi  \haz z_\epsi w_\epsi$$
where $w_\epsi$ is given in \eqref{we}, we readily check that
$\I{+}\epsi  \zeta_\epsi  \in \SL$ and $\zeta_\epsi \to \haz z - z$.

Let now $t \mapsto P_\epsi(t) \in \CC^1(0,1;\Rzn)$ be such that
$P_\epsi(0)=\I$, $P_\epsi(1)=\I {+} \epsi \zeta_\epsi$, and
$$D(I, \I {+} \epsi \zeta_\epsi) \geq (1 {-} \epsi) \int_0^1 R(\dot P_\epsi P_\epsi^{-1})
\dt.$$
Such function $P_\epsi$ exists by the very definition of $D$.
By possibly reparametrizing $P_\epsi$ and using assumption \eqref{R2}
 and bound  \eqref{Dbound} we can assume that 
\begin{equation}\label{1} 
c_4|\dot P_\epsi(t) P_\epsi^{-1}(t)| \stackrel{\eqref{R2}}{\leq} R(\dot P_\epsi(t)
P_\epsi^{-1}(t)) \leq 2 D(I, \I {+} \epsi \zeta_\epsi)  \stackrel{\eqref{Dbound}}{\leq}  c\epsi.
\end{equation}
Hence, 
$$ |P_\epsi(t) {-} \I| \leq \int_0^t
|\dot P_\epsi P_\epsi^{-1}|\, |P_\epsi|\ds \leq  c\epsi\int_0^t
|P_\epsi|\ds \leq c\epsi \left( 1 + \int_0^t
|P_\epsi {-}\I|\ds\right) $$
 so that $P_\epsi \to \I$ uniformly by Gronwall Lemma. 

By defining $\haz P_\epsi(t) = \I + (P_\epsi(t) {-} \I)/\epsi$ one has that $\haz P_\epsi(0)=\I$ and $\haz P_\epsi(1)=\I
{+} \zeta_\epsi$. Moreover, as $\epsi \dot{\haz P_\epsi} 
= \dot P_\epsi$ and $R$ is positively $1$-homogeneous \eqref{R1}, we have that
$$\frac{1}{\epsi}D(\I, \I {+} \epsi \zeta_\epsi) \geq (1 {-} \epsi) \int_0^1
R(\dot{\haz P_\epsi} P_\epsi^{-1}) \dt.$$
Owing now to bound \eqref{1}, by possibly extracting not relabeled
subsequences, we have that $\dot{\haz P_\epsi} \to Q $ weakly-star in $\LL^\infty(0,1;\Rzn)$ and 
\begin{align}
&\liminf_{\epsi \to 0}D_\epsi(z_\epsi,\haz z_\epsi) = \liminf_{\epsi
  \to 0} \frac{1}{\epsi} D(\I, \I {+}\epsi\zeta_\epsi) \nonumber\\
&\geq \liminf_{\epsi \to 0} \int_0^1
R(\dot{\haz P_\epsi} P_\epsi^{-1}) \dt\geq \int_0^1 R(Q) \dt\geq
R(\tilde Q)\nonumber
\end{align}
where we have exploited the lower semicontinuity tool of Lemma \ref{balder}
and used Jensen's inequality with $\tilde Q = \int_0^1 Q\dt$.

Finally, by integrating we have that
$$ \tilde Q = \int_0^1 Q \dt = \lim_{\epsi \to 0}\int_0^1 \dot{\haz P_\epsi} \dt =\lim_{\epsi \to
  0} \zeta_\epsi= \haz z - z$$
so that we have checked 
$$\liminf_{\epsi \to 0} D_\epsi(z_\epsi,\haz z_\epsi) \geq R(\haz z {-}
z).$$

{\it Recovery sequence.} Given $\zeta \in \Rzd$ we have that
$\exp(\zeta) \in \SL$ and, by taking $P(t):= \exp(t \zeta)$ into the definition
  of $D$, we
  readily check that $D(\I,\exp(\zeta)) \leq R(\zeta)$. 

Let now $z,  \, \haz z \in \Rzd$ be given and define
$$\haz z_\epsi = \frac{1}{\epsi} \big( \exp(\epsi (\haz z {-} z))(\I
{+} \epsi z) - \I\big).$$
As $(\I {+}\epsi \haz z_\epsi)(\I{+}\epsi z)^{-1} = \exp(\epsi (\haz z {-} z))$, we have that 
\begin{align}
 \limsup_{\epsi \to 0} D_\epsi(z,\haz z_\epsi) = \limsup_{\epsi \to
   0}\frac{1}{\epsi} D(\I, \exp(\epsi(\haz z {-} z))) \leq R(\haz z
 {-} z) = D_0(z,\haz z)\nonumber
\end{align}
so that $(z,\haz z_\epsi)$ is a recovery sequence.
\end{proof}

Owing to Lemma \ref{GD}, it suffices now to
  apply the lower semicontinuity result in Lemma \ref{balder} in order
  to establish the
$\Gliminf$ inequality for the
dissipation functionals. More precisely, we have following.

\begin{lemma}[$\Gliminf$ for the dissipation]\label{Gdissipation}
  \begin{align}
    \label{GliminfDD}
    \DD_0(z,\haz z) \leq \inf &\Big\{ \liminf_{\epsi \to 0} \DD_\epsi
      (z_\epsi,\haz z_\epsi) \ \big| \nonumber\\
&\  (z_\epsi,\haz
      z_\epsi)\to(z,\haz z) \ \text{weakly in} \ (\LL^2(\Omega;\Rzn))^2  \Big\}.
  \end{align}
\end{lemma}

\subsection{Mutual recovery sequence}\label{sec:mutual}

We now come to the construction
of a {\it mutual recovery sequence}. Let us recall from
\cite{mrs} that indeed {\it two separate} $\Glimsup$ inequalities for 
energy and dissipation generally do not suffice for passing to the 
limit in RIS. In particular, the construction of recovery sequences for
energy and dissipation has to be {\it mutually} coordinated.

\begin{lemma}[Mutual recovery sequence]\label{mutual} Let $t \in [0,T]$, $(u_\epsi, z_\epsi) \to
  (u_0,z_0)$ weakly in $\QQ$, and 
\begin{equation}\sup_{\epsi} \EE_\epsi(t,u_\epsi,z_\epsi) < \infty.\label{enbound}
\end{equation}
Moreover, let $(\haz u_0, \haz z_0):=  (u_0,z_0)
  + (\tilde u, \tilde z)$ with  $(\tilde u, \tilde z) \in
  \CC^\infty_{\rm c}(\Omega;\Rz^d)\times   \CC^\infty_{\rm
    c}(\Omega;\Rzd) $.  Then,
  there exist $(\haz u_\epsi, \haz z_\epsi)\in \QQ$ such that 
  \begin{align}
& (\haz u_\epsi,\haz z_\epsi) \to (\haz u_0,\haz z_0) \ \ \text{weakly
  in} 
\ \QQ   \ \text{and}\nonumber\\
     & \limsup_{\epsi \to 0} \Big( \EE_\epsi(t,\haz u_\epsi, \haz z_\epsi)
    - \EE_\epsi(t, u_\epsi,
    z_\epsi) + \DD_\epsi (z_\epsi,\haz z_\epsi)\Big) \nonumber\\
&\leq  \Big(\EE_0(t,\haz u_0, \haz z_0)
  - \EE_0( t,u_0,  z_0)   + \DD_0 (z_0,\haz z_0)\Big). \label{mrs}
  \end{align}
\end{lemma}

  \begin{proof} For the sake of clarity, we decompose this
    argument into subsequent steps. The general strategy of the proof
    is to choose $(\haz u_\epsi, \haz z_\epsi)$ and show convergence
    to $(\haz u_0,\haz z_0)$, 
$$\limsup_{\epsi \to 0}\DD_\epsi(z_\epsi, \haz z_\epsi) \leq
\DD_0(z_0,\haz z_0) = R(\tilde z),$$
and
$$\limsup_{\epsi \to 0} \Big( \EE_\epsi(t,\haz u_\epsi,\haz z_\epsi) -
\EE_\epsi(t,u_\epsi,z_\epsi)\Big) \leq \EE_0(t,\haz u_0,\haz z_0) -
\EE_0(t,u_0,z_0).$$
Note that in order to establish the latter we cannot argue on individual terms
but rather aim at exploiting certain cancellations. This
resembles the situation of the so-called {\it quadratic trick} (see,
e.g., \cite{Mielke-Timofte05}) and crucially uses \eqref{Mandel} as
well as the smoothness of $(\tilde u, \tilde z)$. 
In particular, note that within this proof the
    constant $c$ may depend on $\tilde u$ and $\tilde z$ as well.

{\it Step 1: Choice of the mutual recovery sequence.}
By defining the functions $\psi_\epsi := \id + \epsi \tilde u$ and $\varphi_\epsi := \id +
\epsi u_\epsi$ and the set 
$$ \Omega_\epsi := \big\{x \in \Omega \ \big|\ \exp(\epsi \tilde z(x)) (\I {+}
\epsi z_\epsi(x))  \in K \big\},$$
the proof of the lemma follows by checking that the choices
\begin{align}
 \haz u_\epsi&:= \frac{1}{\epsi} \big(
  \psi_\epsi \circ \varphi_\epsi {-} \id \big),\nonumber\\
\haz z_\epsi &:= 
\left\{
  \begin{array}{ll}
\disp \frac{1}{\epsi}\big(
  \exp(\epsi \tilde z) (\I {+} \epsi z_\epsi)- \I \big)\quad
  &\text{on} \ \  \Omega_\epsi \\
z_\epsi &\text{else},
\end{array}
\right.\nonumber
\end{align}
fulfill \eqref{mrs}  and, in particular, $(\haz u_\epsi, \haz z_\epsi) \to
(\haz u_0, \haz z_0)$ weakly in $\QQ$.   The construction of $\haz u_\epsi$ via a
composition and of $\haz z_\epsi$ via matrix exponential and
multiplication is necessary in order to deal with the multiplicative nature 
of finite-strain elastoplasticity.

Note that the construction of the mutual recovery sequence is
compatible with the constraint ${\rm det} (\I{+}\epsi\nabla \haz u_\epsi
)>0$  considered in \eqref{smoothness}.  Indeed, by letting $\epsi$ be small enough we have
that $\I {+}\epsi \nabla \tilde u$ is everywhere positive definite,
hence $\det (\I {+}\epsi \nabla \tilde u)>0$. In particular, as $\det
(\I{+}\epsi \nabla u_\epsi)>0$ almost everywhere  by
\eqref{smoothness} and \eqref{enbound},  we have that 
$$\det (\I {+} \epsi \nabla \haz u_\epsi) =
\det(\nabla \psi_\epsi (\varphi_\epsi) \nabla \varphi_\epsi) = 
\det(\I{+}\epsi\nabla\tilde u(\varphi_\epsi))\det(\I {+}\epsi \nabla u_\epsi)>0
$$
almost everywhere as well. That is, $\I{+}\epsi \nabla \haz u_\epsi \in \GLp$ almost
everywhere.

From the bound \eqref{enbound} we readily have
that $\I{+}\epsi z_\epsi \in \SL$ almost everywhere. Hence, upon noting that 
$$\I{+}\epsi \haz z_\epsi = 
\left\{
  \begin{array}{ll}
  \exp(\epsi \tilde z) (\I {+} \epsi z_\epsi)\quad
  &\text{on} \ \  \Omega_\epsi \\
\I{+}\epsi z_\epsi &\text{else},
\end{array}
\right.
$$
we immediately check that $(\I {+} \epsi \haz z_\epsi) \in K \subset
\SL$ 
almost everywhere and is  bounded in $\LL^\infty$.
Using the fact that ${\rm tr}\, \tilde z = 0$ we have
${\rm det} \exp(\epsi \tilde z)=\exp ( \epsi  {\rm tr} \, \tilde z)=1$ and hence $ \exp(\epsi \tilde
z)(\I{+}\epsi z_\epsi) \in \SL$ almost everywhere. 

Next, note that the measure of the complement of $ \Omega_\epsi $ can be controlled 
by means of a Chebyshev estimate. Indeed, relation
\eqref{conscomp2} gives
\begin{align}
|\Omega \setminus  \Omega_\epsi | &= \int_{\Omega \setminus  \Omega_\epsi } 1 \dx \leq c_K^2 \int_\Omega \big|\exp(\epsi \tilde z)
(\I {+} \epsi z_\epsi) -\I\big|^2 \dx\nonumber\\
&  
= c_K^2\int_\Omega \big|\exp(\epsi \tilde z) - \I +
\epsi \exp (\epsi \tilde z) z_\epsi\big|^2 \dx \leq c \epsi^2 \left(1 {+}
  \int_\Omega z_\epsi^2 \dx\right) \leq c\epsi^2. \nonumber
\end{align}
Now, one has that 
\begin{align*}
  \haz z_\epsi - z_\epsi &= \frac{1}{\epsi} \big( \exp(\epsi \tilde
  z)(\I{+}\epsi z_\epsi) - \I\big) - z_\epsi=
  \frac{1}{\epsi}(\exp(\epsi \tilde z) {-} \I )(\I{+}\epsi z_\epsi)
  \quad \text{on} \ \  \Omega_\epsi ,\\
  \haz z_\epsi - z_\epsi &= 0 \quad \text{on} \ \ \Omega \setminus  \Omega_\epsi ,
\end{align*}
the convergence $|\Omega \setminus  \Omega_\epsi |\to 0$, and that $\haz z_\epsi$ and $z_\epsi$ are
 bounded in  $\LL^2$. Hence, we readily check that 
\begin{align}
 &\haz z_\epsi - z_\epsi \to \tilde z \ \ \text{strongly in} \ \LL^2(\Omega;\Rzn).  \label{convz}
\end{align}
 This implies that $\haz z_\epsi \to \haz z_0 = z_0{+} \tilde z$
weakly in $\LL^2$, hence 
\begin{align}
&\haz z_\epsi + z_\epsi \to \haz z_0 + z_0 \ \ \text{weakly in} \ \LL^2(\Omega;\Rzn),  \label{convz0}.
\end{align}

From the energy bound \eqref{enbound} and the coercivity Lemma \ref{encoerc} 
we have that $u_\epsi$ is  bounded in $\HH^1$  and $\epsi
u_\epsi \to 0$  strongly in $\LL^2$. Hence, one has that
$\|\varphi_\epsi{-} \id\|_{\LL^2} = \epsi \|u_\epsi\|_{\LL^2 } \leq
c\epsi$
and, by the Lipschitz continuity of $\nabla \tilde u$, we conclude that
\begin{equation}
\|\nabla \tilde u(\varphi_\epsi) -\nabla \tilde u \|_{\LL^2} \leq
c\|\varphi_\epsi{-} \id\|_{\LL^2} = c\epsi \|u_\epsi\|_{\LL^2 }\leq c
\epsi.\label{use1}
\end{equation}
Moreover, by computing
\begin{align}
 \nabla \haz u_\epsi &= \frac{1}{\epsi} \big( \nabla
 \psi_\epsi(\varphi_\epsi) \nabla \varphi_\epsi {-} \I\big) 
 = \frac{1}{\epsi}
 \big( (\I{+}\epsi \nabla \tilde u)(\varphi_\epsi) \nabla \varphi_\epsi {-}
 \I\big)  \nonumber\\
&= \frac{1}{\epsi}
 \big( \nabla \varphi_\epsi {+} \epsi \nabla \tilde u(\varphi_\epsi) \nabla \varphi_\epsi {-}
 \I\big) = \nabla u_\epsi +  \nabla \tilde u(\varphi_\epsi)  + \epsi  \nabla
\tilde u(\varphi_\epsi)  \nabla u_\epsi\nonumber
\end{align}
we obtain that 
\begin{align}
\|(\nabla \haz u_\epsi {-} \nabla u_\epsi) - \nabla \tilde u\|_{\LL^2}
&\leq \| \nabla \tilde u (\varphi_\epsi) {-} \nabla \tilde u\|_{\LL^2} +
 \| \epsi \nabla \tilde u (\varphi_\epsi) \nabla u_\epsi\|_{\LL^2}
\nonumber\\
& \stackrel{\eqref{use1}}{\leq} c \epsi+ c \epsi \| \nabla
u_\epsi\|_{\LL^2}\leq c\epsi\label{nablau}
\end{align}
 and this implies that $\haz u_\epsi \to \haz u_0 = u_0{+}\tilde
u$ weakly in $\HH^1$. 

The tensors $A_\epsi = (F_{\rm el,\epsi} {-}\I)/\epsi$ and $\haz A_\epsi =
(\haz F_{\rm el,\epsi} {-}\I)/\epsi$ fulfill
\begin{align}
 A_\epsi = \frac{1}{\epsi} \big((\I{+}\epsi \nabla
 u_\epsi)(\I{+}\epsi z_\epsi)^{-1} -\I \big), \quad
  \haz A_\epsi = \frac{1}{\epsi} \big((\I{+}\epsi \nabla
 \haz u_\epsi)(\I{+}\epsi \haz z_\epsi)^{-1} -\I \big)\nonumber
\end{align}
and are hence both bounded in $\LL^2$  by \eqref{conscomp}. 

Fix now $\delta$ and let $c_{\rm el}(\delta)$ and $c_{\rm h}(\delta)$ be given
by conditions \eqref{approx} and \eqref{approx2}, respectively. For
all $\epsi>0$ we define the sets
\begin{align}
U_\epsi^\delta&:= \Big\{x \in \Omega \ \big| \ |\epsi A_\epsi(x)| +  |\epsi
\haz A_\epsi(x)| \leq c_{\rm el}(\delta) \Big\}, \nonumber\\
Z_\epsi^\delta&:= \Big\{x \in \Omega \ \big| \ |\epsi z_\epsi(x)| +  |\epsi \haz z_\epsi(x)|
\leq c_{\rm h}(\delta) \Big\},\nonumber
\end{align}
We refer to the latter as {\it good sets} as strains are there
under control and we can replace
the nonlinear densities $\Wel$ and $\Wh$ by their quadratic
expansions via \eqref{approx} and \eqref{approx2}. In particular, on the good sets
the quadratic character of the expansions will entail the control of
the difference of the energy contributions by means of a suitable
cancellation ({\it quadratic trick}). On the other hand, we
term {\it bad sets} the corresponding complements $\Omega\setminus U^\delta_\epsi$ and
$\Omega \setminus Z^\delta_\epsi$ where the quadratic expansions are a priori
not available. Using some nontrivial cancellations, we will show that the difference of the energy
contributions on the bad sets is infinitesimal. Note
preliminarily that the integrands on the bad sets blow up while the bad sets have small measure. Indeed,
\begin{align}
  |\Omega \setminus U_\epsi^\delta| &= \int_{\Omega \setminus
    U_\epsi^\delta }1 \dx \leq \frac{ \epsi^2 }{c_{\rm el}^2(\delta) }
\int_\Omega (|A_\epsi|+|\haz A_\epsi|)^2 \dx  \leq
\frac{c\epsi^2}{c_{\rm el}^2(\delta)}, \label{cU}\\
  |\Omega \setminus Z_\epsi^\delta| &= \int_{\Omega \setminus
    Z_\epsi^\delta }1 \dx \leq \frac{ \epsi^2 }{ c_{\rm h}^2(\delta) }
\int_\Omega (|z_\epsi|+|\haz z_\epsi|)^2 \dx  \leq
\frac{c\epsi^2}{c_{\rm h}^2(\delta)}. \label{cZ}
\end{align}

{\it Step 2: Treatment of the dissipation term.} 
As $\haz z_\epsi = z_\epsi$ on $\Omega \setminus   \Omega_\epsi $ one has that 
\begin{align}
 \DD_\epsi (z_\epsi,\haz z_\epsi) 
& =  \frac1\epsi  \int_{  \Omega_\epsi }   D    (\I,\exp(\epsi \tilde z))\dx
\leq   \frac1\epsi   \int_{\Omega}  D   (\I,\exp(\epsi \tilde
z))\dx.\label{perdopo}
\end{align}

 In the construction of the recovery sequence in the proof of
Lemma \ref{GD} we have proved  that 
\begin{equation}\label{point}  \limsup_{\epsi \to 0} \frac1\epsi D   (\I, \exp(\epsi
  \tilde z)) \leq  R(\tilde z) .
\end{equation}
Eventually, by taking the $\limsup$ in relation \eqref{perdopo} and using 
\eqref{point} we have
proved that  
\begin{align}
 \limsup_{\epsi \to 0} \DD_\epsi ( z_\epsi, \haz z_\epsi)  &= 
  \limsup_{\epsi \to 0} \frac1\epsi   \int_\Omega  D (\I,\exp(\epsi \tilde
z))\dx \nonumber\\
&\leq   \int_\Omega R (\tilde z) \dx = \DD_0(z_0,\haz z_0).\label{D0}
\end{align}


{\it Step 3: Limsup for the differences of the elastic energy terms.}
Let us start 
 by rewriting the tensors $A_\epsi$ as
\begin{align}
A_\epsi=\nabla u_\epsi -
z_\epsi +w_\epsi - \epsi \nabla u_\epsi z_\epsi + \epsi \nabla u_\epsi w_\epsi\nonumber
\end{align}
 where $w_\epsi$ is given by \eqref{we}.  On the other hand, as regards the tensors $\haz A_\epsi$ we have that 
\begin{align}
\haz A_\epsi&=\frac{1}{\epsi}
\big(  (\I +\epsi\nabla \haz u_\epsi)(\I {-} \epsi  z_\epsi{+} 
\epsi 
w_\epsi)\exp(-\epsi \tilde z) -\I \big)\nonumber\\
&= \big(\nabla
\haz u_\epsi {-} z_\epsi {+} w_\epsi {-}\epsi \nabla \haz u_\epsi
z_\epsi + \epsi \nabla \haz u_\epsi w_\epsi 
\big)\exp(-\epsi \tilde z) \nonumber\\
&\quad+ \frac{1}{\epsi}\big( \exp(-\epsi \tilde z)
{-}\I\big)\quad \text{on} \ \  \Omega_\epsi \nonumber\\
\haz A_\epsi&=\frac{1}{\epsi}
\big(  (\I +\epsi\nabla \haz u_\epsi)(\I {-} \epsi  z_\epsi{+} 
\epsi 
w_\epsi)-\I \big)\nonumber\\
&= \nabla
\haz u_\epsi {-} z_\epsi {+} w_\epsi {-}\epsi \nabla \haz u_\epsi
z_\epsi {+} \epsi \nabla \haz u_\epsi w_\epsi 
 \quad \text{on} \ \ \Omega \setminus   \Omega_\epsi .\nonumber
\end{align}
Hence, one can compute that  
\begin{align}
 \haz A_\epsi - A_\epsi &= (\nabla \haz u_\epsi {-} \nabla
 u_\epsi)(\I{-}\epsi z_\epsi{+}\epsi w_\epsi) +
 \frac{1}{\epsi}\big(\exp(-\epsi \tilde z) {-} \I\big) \nonumber\\
&\quad +(\nabla \haz u_\epsi {-}z_\epsi{+}w_\epsi{-}\epsi \nabla \haz
u_\epsi z_\epsi{+} \epsi \nabla \haz u_\epsi w_\epsi) (\exp(-\epsi
\tilde z) {-} \I) \quad \text{on} \ \  \Omega_\epsi \nonumber\\ 
\haz A_\epsi - A_\epsi &= (\nabla \haz u_\epsi {-} \nabla
 u_\epsi)(\I{-}\epsi z_\epsi{+}\epsi w_\epsi) \quad \text{on} \ \
 \Omega\setminus  \Omega_\epsi .\nonumber
\end{align}
In particular, owing to convergence \eqref{nablau}  and the
$\LL^\infty$ bounds for $\epsi z_\epsi$ and $\epsi w_\epsi$ (see the
discussion after \eqref{we})  we have that $(\nabla \haz u_\epsi {-} \nabla
 u_\epsi)(\I{-}\epsi z_\epsi{+}\epsi w_\epsi)$ converges to $\nabla \tilde u$
 strongly in $\LL^2$. Thus,  by recalling that $w_\epsi \to 0$
 weakly in $\LL^2$  
 it is a standard matter to check that 
\begin{align}
&\haz A_\epsi + A_\epsi \to  (\nabla\haz u_0 {-} \haz z_0) +
(\nabla u_0 {-}  z_0) \ \ \text{weakly in} \ \ \LL^2(\Omega; \Rzn),\label{conv1}\\
&\haz A_\epsi - A_\epsi \to  \nabla \tilde u  - \tilde z \ \ \text{strongly in} \ \ \LL^2(\Omega; \Rzn).\label{conv2}
\end{align}

On the good set $U_\epsi^\delta$ we will use the assumption \eqref{approx}
in order to have that 
\begin{align}
 &\Wel^\epsi(\haz A_\epsi) -  \Wel^\epsi( A_\epsi) \leq
 |\haz A_\epsi|^2_\C - | A_\epsi|^2_\C + \delta
 \big(|\haz A_\epsi|^2_\C + |A_\epsi|_\C^2 \big)  \nonumber\\
&= \frac12
 (\haz A_\epsi {-} A_\epsi){:}\C (\haz A_\epsi {+} A_\epsi) + \delta
 \big(|\haz A_\epsi|^2_\C + |A_\epsi|_\C^2 \big) .\label{qq}
\end{align}

Let us now argue on the bad set $\Omega \setminus U^\delta_\epsi$ by defining
$$ G_{1,\epsi}:= (\I{+}\epsi \nabla \haz u_\epsi) (\I {+} \epsi \nabla
u_\epsi)^{-1}, \quad G_{2,\epsi}:= (\I {+} \epsi  z_\epsi) (\I
{+} \epsi \haz z_\epsi)^{-1}.$$
 The energy bound \eqref{enbound}, together with assumption
\eqref{smoothness}, implies that $\I{+}\epsi \nabla u_\epsi$ is
invertible almost everywhere.  
Note that $G_{1,\epsi}$ and $G_{2,\epsi}$ are
chosen in such a way that $\haz F_{\el,\epsi} = G_{1,\epsi}F_{\el,\epsi} G_{2,\epsi}$.
We readily compute that 
\begin{align}
 G_{1,\epsi} -\I&= \nabla \psi_\epsi(\varphi_\epsi)\nabla \varphi_\epsi (\I {+} \epsi \nabla
u_\epsi)^{-1} - \I= \nabla \psi_\epsi(\varphi_\epsi) -\I = \epsi \nabla
\tilde u (\varphi_\epsi)\nonumber
\end{align}
so that $ \|G_{1,\epsi} {-} \I\|_{\LL^\infty(\Omega \setminus U^\delta_\epsi;\Rzn)} = \epsi
\|\nabla \tilde u (\varphi_\epsi)\|_{\LL^\infty(\Omega \setminus U^\delta_\epsi;\Rzn)} \leq
c\epsi$. Moreover, one has that
$$G_{2,\epsi}= 
\left\{
\begin{array}{ll}
\exp(-\epsi \tilde z) \quad &\text{on}  \ (\Omega \setminus
U^\delta_\epsi) \cap  \Omega_\epsi ,\\
\I & \text{on} \ \Omega \setminus (U^\delta_\epsi \cup
\Omega_\epsi). 
\end{array}
\right.
$$
Hence, $ \| G_{2,\epsi} - \I \|_{\LL^\infty(\Omega \setminus U^\delta_\epsi;\Rzd)} \leq
c\epsi$ as well. Next, estimate \eqref{key} and bound \eqref{enbound}
allow us to control the elastic part of the energy on the bad set
$\Omega \setminus U^\delta_\epsi$ (where $\nabla u_\epsi$ and
$z_\epsi$ are not under control) by cancellation. For this we employ the
multiplicative estimate \eqref{Mandel} provided in \eqref{key}:
\begin{align}
&\int_{\Omega \setminus U^\delta_\epsi} 
\big(\Wel^\epsi(\haz A_\epsi)-\Wel^\epsi(A_\epsi)\big) \dx =  \frac{1}{\epsi^2}\int_{\Omega \setminus U^\delta_\epsi} \big(\Wel(\haz F_{\el,\epsi})-\Wel(F_{\el,\epsi})\big) \dx\nonumber\\
&=
 \frac{1}{\epsi^2} \int_{\Omega \setminus U^\delta_\epsi} \big(
  \Wel(G_{1,\epsi} F_{\el,\epsi} G_{2,\epsi}) -
 \Wel(F_{\el,\epsi}) \big)  \dx
 \nonumber\\
& \stackrel{\eqref{key}}{\leq} \frac{c_7}{\epsi^2} \int_{\Omega \setminus U^\delta_\epsi} \big(\Wel (F_{\el,\epsi}) + c_8\big) \big( |G_{1,\epsi} {-}
\I| + | G_{2,\epsi} {-} \I|\big)\dx\nonumber\\
& \leq c_7 \left( \frac{1}{\epsi^2} \int_{\Omega} \Wel (F_{\el,\epsi})
  \dx  + \frac{c_8}{\epsi^2}|\Omega \setminus U^\delta_\epsi|\right)\big(\|G_{1,\epsi} {-}
\I\|_{\LL^\infty} + \| G_{2,\epsi} {-}
\I\|_{\LL^\infty} \big) \nonumber\\
&\stackrel{ \eqref{enbound} \&  \eqref{cU}}{\leq} c \left(
  1{+} \frac{1}{c_{\rm el}^2(\delta)}\right) \epsi\label{key_conv}.
\end{align}
Thus, we have controlled the difference of the energy contributions in
the bad set $\Omega\setminus U^\delta_\epsi$ where the gradients are big.

Finally, by using 
convergences \eqref{conv1}-\eqref{conv2}, equation \eqref{qq} on
the good set $U^\delta_\epsi $, relation
\eqref{key_conv} on the bad set $\Omega \setminus U^\delta_\epsi$,
 and the $\LL^2$ boundedness of $\haz A_\epsi$ and $A_\epsi$, 
we conclude that
\begin{align}
&\limsup_{\epsi \to 0}\left( \int_{\Omega} \Wel^\epsi(\haz A_\epsi)\dx -
 \int_{\Omega} \Wel^\epsi(A_\epsi)\dx \right) \nonumber\\
&\stackrel{\eqref{qq}}{\leq} \limsup_{\epsi \to 0}\Bigg(  \frac12  \int_{U_\epsi^\delta} \big(\haz A_\epsi {-} A_\epsi\big){:}\C
\big(\haz A_\epsi {+} A_\epsi\big) \dx  +  c\delta  \nonumber\\
&\hspace{23mm}+ \int_{\Omega\setminus U^\delta_\epsi}
\big(\Wel^\epsi(\haz A_\epsi)- \Wel^\epsi(A_\epsi)\big) \dx
\Bigg)\nonumber\\
& \stackrel{\eqref{key_conv} }{ \leq } \limsup_{\epsi\to  0 }\Bigg(  \frac12  \int_{U_\epsi^\delta} \big(\haz A_\epsi {-} A_\epsi\big){:}\C
\big(\haz A_\epsi {+} A_\epsi\big) \dx  +  c\delta  +
c\left(1{+} \frac{1}{c_{\rm el}^2(\delta)} \right) \epsi\Bigg)\nonumber\\
& \leq  \frac12\int_\Omega \big( \nabla \tilde u - \tilde z \big){:}\C \big( \nabla (\haz u_0 {+} u_0) - (\haz z_0  {+}
z_0) \big) \dx  +  c\delta \nonumber\\
& = \int_\Omega |\nabla \haz u_0^\sy  {-} \haz z_0^\sy|_\C^2 \dx -
\int_\Omega |\nabla u_0^\sy  {-}  z_0^\sy|_\C^2 \dx  + c\delta\label{E0}
\end{align}
where we have made use of relation \eqref{minor}.

{\it Step 4: Upper bound on the hardening energy term.}
Let us now turn our attention to the
hardening part of the energy. On the good set $Z_\epsi^\delta$ we have that
\begin{align}
 \Wh^\epsi(\haz z_\epsi) - \Wh^\epsi(z_\epsi) &\leq |\haz z_\epsi|^2_\H  -
 |z_\epsi|^2_\H+  \delta
 \big(|\haz z_\epsi|^2_\H + |z_\epsi|_\H^2 \big)\nonumber\\
&=\frac12
 (\haz z_\epsi{-}z_\epsi){:}\H(\haz z_\epsi{+}z_\epsi) +  \delta
 \big(|\haz z_\epsi|^2_\H + |z_\epsi|_\H^2 \big).  \label{qq2}
\end{align}
As regards the bad set $\Omega\setminus Z_\epsi^\delta$ one has
that 
\begin{align}
 \Wh^\epsi(\haz z_\epsi) {-} \Wh^\epsi(z_\epsi)
=
\left\{
\begin{array}{ll}
    \disp\frac{1}{\epsi^2}\Whf(\exp(\epsi \tilde z) (\I{+}\epsi z_\epsi)) -
    \frac{1}{\epsi^2}\Whf( \I{+}\epsi z_\epsi)   &\text{on} \ 
    (\Omega\setminus Z_\epsi^\delta)\cap  \Omega_\epsi \\0  \
    &\text{on} \
\Omega\setminus (Z_\epsi^\delta \cup  \Omega_\epsi ).
\end{array}
\right.\nonumber
\end{align}
Hence, by exploiting the local Lipschitz continuity of
$\Whf$ we have that 
\begin{align}
&\int_{\Omega\setminus Z_\epsi^\delta} \big(\Wh^\epsi(\haz
  z_\epsi)-  \Wh^\epsi(z_\epsi) \big) \dx
 \leq \frac{c}{\epsi^2} \int_{\Omega\setminus Z_\epsi^\delta} |\exp(\epsi
\tilde z) - \I|\, |\I {+}\epsi z_\epsi| \dx \nonumber\\
& \leq \frac{c}{\epsi^2} |\Omega\setminus
Z_\epsi^\delta| \, \frac{c\epsi}{ c_{\rm h}^2(\delta) }
\stackrel{\eqref{cZ}}{\leq} \frac{c\epsi}{ c_{\rm h}^2(\delta) }.\label{key_conv2}
\end{align}

Eventually, owing to convergences  \eqref{convz}-\eqref{convz0}
 we compute that 
\begin{align}
 &\limsup_{\epsi \to 0}\left(\int_\Omega \Wh^\epsi(\haz z_\epsi)\dx
   - \int_\Omega \Wh^\epsi(z_\epsi) \dx \right)\nonumber\\
&\stackrel{\eqref{qq2}}{\leq}  \limsup_{\epsi \to
  0}\Bigg(\int_{Z_\epsi^\delta} \frac12 (\haz
  z_\epsi{-}z_\epsi){:}\H(\haz z_\epsi{+}z_\epsi) \dx + c\delta
    \nonumber\\
&\hspace{23mm}+\int_{\Omega\setminus Z_\epsi^\delta} \big(\Wh^\epsi(\haz
  z_\epsi)-  \Wh^\epsi(z_\epsi) \big) \dx
\Bigg)\nonumber\\
& \stackrel{\eqref{key_conv2}}{ \leq}  \limsup_{\epsi \to
  0}\left(\int_{Z_\epsi^\delta} \frac12 (\haz
  z_\epsi{-}z_\epsi){:}\H(\haz z_\epsi{+}z_\epsi) \dx + c\delta
   +  \frac{c\epsi}{c_{\rm el}^2(\delta)}\right)\nonumber\\
&= \int_{\Omega} \frac12 \tilde z{:}\H(\haz z_0{+}z_0) \dx  + c\delta
    = \int_{\Omega}|\haz z_0|^2_\H \dx-\int_{\Omega}| z_0|^2_\H \dx + c\delta.
   \label{E02}
\end{align}

{\it Step 5: Conclusion of the proof.} By collecting relations 
\eqref{E0} and \eqref{E02}, and recalling that $\lan \ell(t), u_\epsi -
\haz u_\epsi\ran \to \lan \ell(t), u_0 -
\haz u_0\ran $ we have proved that 
\begin{align}
 &\limsup_{\epsi \to 0}\Big(\EE_\epsi (t,\haz u_\epsi, \haz z_\epsi) {-}
 \EE_\epsi(t,u_\epsi,z_\epsi)\Big) \leq \Big(\EE_0 (t,\haz u_0, \haz z_0) {-}
 \EE_0(t,u_0,z_0)
 \Big) +  c\delta . \nonumber
\end{align}
Finally, the assertion \eqref{mrs} follows by taking $\delta \to 0$
and employing \eqref{D0}.
\end{proof}

\subsection{Proof of Theorem {\ref{main}}}\label{sec:proof}

Owing to the the above-obtained $\Gliminf$ and mutual-recovery-sequence results, the proof of Theorem \ref{main} now follows along 
the lines of the general theory of \cite{mrs}. We limit
ourselves  to sketch  the main points of the argument and refer the
reader to \cite{mrs} for the details.

Let $(u_\epsi,z_\epsi)$ be a sequence of finite-plasticity solutions. 
The coercivity of the energy \eqref{bound} entails an a priori bound
on $(u_\epsi,z_\epsi)$. In particular, we have the following.

\begin{corollary}[A priori bound] There exists $c>0$ such that all
  finite-plasticity solutions $(u_\epsi,z_\epsi)$ fulfill
  \begin{equation}
    \label{bound2}
  \forall t \in [0,T]: \quad  \| u_\epsi(t)\|_{\HH^1}+ \|z_\epsi(t)\|_{\LL^2} + \| \epsi
    z_\epsi(t)\|_{\LL^\infty} + {\rm
      Diss}_{\DD_\epsi}(z_\epsi;[0,t]) \leq c.
  \end{equation}
\end{corollary}

\begin{proof}
We exploit the energy balance \eqref{energy} and the bound \eqref{bound}
in order to get that, for all $t\in [0,T]$, 
\begin{align}
 & \|\nabla u_\epsi(t) \|^2_{\LL^2} +  \| z_\epsi(t)\|^2_{\LL^2} + \| \epsi
 z_\epsi(t)\|^2_{\LL^\infty} +{\rm
   Diss}_{\DD_\epsi}(z_\epsi;[0,t])\nonumber\\
& \stackrel{\eqref{bound}}{\leq}  c\big(1{+}
 \WW_\epsi(u_\epsi(t),z_\epsi(t))\big) +{\rm Diss}_{\DD_\epsi}(z_\epsi;[0,t])\nonumber\\
&\leq c\big( 1 + \EE_\epsi(t,u_\epsi(t),z_\epsi(t)) + \lan \ell(t),
 u_\epsi (t)\ran
+{\rm Diss}_{\DD_\epsi}(z_\epsi;[0,t]) \big)
\nonumber\\
& \stackrel{\eqref{energy}}{=} c \left(1+
\EE_\epsi(0,u^0_\epsi,z^0_\epsi)  + \lan \ell(t),  u_\epsi  (t) \ran- \int_0^t \lan
\dot \ell,  u_\epsi  \ran \ds\right)\nonumber \\
&\leq c \left(1+
\EE_\epsi(0,u^0_\epsi,z^0_\epsi)  + \|
\ell(t)\|_{\HH^{-1}}\| u_\epsi  (t)\|_{\HH^1} + \int_0^t \| \dot \ell\|_{\HH^{-1}}\| u_\epsi \|_{\HH^1} \ds\right)\nonumber 
\end{align}
so that the assertion follows by Gronwall Lemma.
\end{proof}

Owing to the a priori bound \eqref{bound2}, we may now exploit the
generalized version of Helly's Selection Principle in
\cite[Thm. A.1]{mrs} (consider also the comments thereafter) and deduce
that, at least for some nonrelabeled subsequence, and all $s,\,t \in
[0,T]$ with $s<t$,
\begin{align}
& \delta_0(t):= \lim_{\epsi \to 0} {\rm Diss}_{\DD_\epsi}(z_\epsi;[0,t]),\nonumber\\
  & z_\epsi(t) \to z_0(t) \quad \text{weakly
    in} \ \ \ZZ,\nonumber\\
& {\rm
  Diss}_{\DD_0}(z_0;[s,t]) \leq\delta_0 (t) - \delta_0(s),\nonumber 
\end{align}

Moreover, by letting $t\in[0,T]$ be fixed we may extract a further
subsequence (still not relabeled, possibly depending on $t$) such that 
$u_\epsi(t)\to u_*$ weakly in $\UU$. 
We now check that indeed
$(u_*,z_0(t)) \in \SS_0(t)$. To this aim, by density it suffices to
consider competitors $(\haz
u_0,\haz z_0) =(u_*,z_0(t))   +  (\tilde u,\tilde z) $ with
$(\tilde u, \tilde z)$ smooth and compactly supported. By applying
Lemma \ref{mutual} we find
a mutual recovery sequence $(\haz u_\epsi, \haz
z_\epsi)$ such that 
\begin{align}
 & \EE_0(t,\haz u_0,\haz z_0) - \EE_0(t,u_*,z_0(t)) + \DD_0(z_0(t),\haz z_0)
  \nonumber\\
&\geq \limsup_{\epsi\to 0}\left( \EE_\epsi(t,\haz u_\epsi,\haz
    z_\epsi) - \EE_\epsi(t,u_\epsi(t),z_\epsi(t)) + \DD_\epsi(z_\epsi(t),\haz
    z_\epsi)  \right) \geq 0\label{crux}
\end{align}
where the last inequality follows from the stability \eqref{stability}
of 
$(u_\epsi(t),z_\epsi(t))$. Hence, we have proved that $(u_*,z_0(t))\in
\SS_0(t)$. Note that, given $z_0(t) \in \ZZ$, as the
functional $ u \in \UU \mapsto \EE_0(t,u,z_0(t))$ is uniformly convex
 there exists a
    unique $u_0(t) \in \UU$ such that $(u_0(t),z_0(t))\in
    \SS_0(t)$. From the fact that $(u_*,z_0(t)) \in \SS_0(t)$ we
    conclude that $u_*\equiv u_0(t)$. In particular $u_\epsi (t) \to
    u_0(t)$ weakly in $\UU$
    for all $t \in [0,T]$ and the whole sequence converges.

 Let now be given a partition $\{0 = t_0<
t_1< \dots< t_N=t\}$. By passing to the $\liminf$ in the
energy balance \eqref{energy} and using Lemmas \ref{Genergy} and
\ref{Gdissipation} we get that
\begin{align}
 &\EE_0(t,u_0(t),z_0(t)) + \sum_{i=1}^N\DD_0(z_0(t_i),z_0(t_{i-1}))
  \nonumber\\
&\leq \liminf_{\epsi \to 0}\left( \EE_\epsi(t,u_\epsi(t),z_\epsi(t)) +
  \sum_{i=1}^N\DD_\epsi(z_\epsi(t_i),z_\epsi(t_{i-1}))
\right)\nonumber\\
&\leq \liminf_{\epsi \to 0}\left( \EE_\epsi(0,u^0_\epsi,z^0_\epsi) -
  \int_0^t \lan \dot \ell, u_\epsi \ran \ds\right)=
\EE_0(0,u^0_0,z^0_0) - \int_0^t \lan \dot \ell, u_0 \ran \ds\nonumber
\end{align}
where for the last equality we have used \eqref{initial_data} and the convergence of
$u_\epsi$.
Hence, the upper energy estimate follows by taking the $\sup$ among
all partitions of the interval $[0,t]$.  The lower energy estimate can
classically recovered from stability as in
\cite[Prop. 2.7]{Mielke05}. This proves that $(u_0,z_0)$ is a linearized-plasticity solution. In particular, as linearized-plasticity solutions are
unique,  the whole sequence $(u_\epsi,z_\epsi)$
converges and no extraction of subsequences is actually needed.

Along the lines of the proof of Theorem \ref{main} (see also
\cite[Thm. 3.1]{mrs}) we also obtain the following convergences.

\begin{corollary}[Improved convergences] Under the assumptions of
  Theorem \emph{\ref{main}} we have that, for all $t \in [0,T]$,\label{ref}
  \begin{align}
  \int_\Omega \big(\Wel^\epsi(A_\epsi) + \Wh^\epsi(z_\epsi) \big)\dx
   & \to \int_\Omega \big(|\nabla u_0{-} z_0|^2_\C +
   |z_0|^2_\H\big)\dx,\label{energy_conv}\\
\Diss_{\DD_\epsi}(z_\epsi;[0,t])& \to \int_0^t R(\dot z)\ds.\label{diss_conv}
  \end{align}
\end{corollary}

In particular, owing to the energy convergence \eqref{energy_conv} we
are in the position of deducing some strong convergence of
finite-plasticity solutions to linearized-plasticity solutions.

\begin{corollary}[Strong convergence]
  Under the assumptions of Theorem \emph{\ref{main}} for all $t \in [0,T]$ we have that $(u_\epsi(t),z_\epsi(t))
\to (u_0(t),z_0(t))$ strongly in ${\rm W}^{1,p}(\Omega; \Rz^d)\times
\LL^p(\Omega;\Rzn)$ for all $p \in [1,2)$.\label{ref2}
\end{corollary}

\begin{proof}
Let $\nu$ denote the Young measure generated by the sequence
$(A_\epsi,z_\epsi)$ and define the measure $\nu^\sy(A^s,Z) := \nu(A^s {\oplus}
\Rzn_\an,Z)$ for all Borel sets $(A^s,Z )\subset \Rzs \times
\Rzn$. Note that $\nu^\sy$ is indeed the Young measure
generated by $(A_\epsi^\sy,z_\epsi)$. By using the lower
semicontinuity Lemma \ref{balder} and the energy convergence
\eqref{energy_conv} we deduce that
\begin{align}
&  \int_\Omega \left( \int_{\Rzs\times\Rzn} \big(|A^\sy|^2_\C + |z|^2_\H\big)
    {\rm d} \nu^\sy_x(A^\sy,z)\right)\dx \nonumber\\
&=  \int_\Omega \left( \int_{\Rzn\times\Rzn} \big(|A|^2_\C + |z|^2_\H\big)
    {\rm d} \nu_x(A,z)\right)\dx \nonumber\\
&\leq \liminf_{\epsi\to 0}\int_\Omega \big(\Wel^\epsi(A_\epsi) + \Wh^\epsi(z_\epsi) \big)\dx \stackrel{\eqref{energy_conv}}{=} \int_\Omega \big(|\nabla u_0 {-} z_0|^2_\C + |z_0|^2_\H\big)
    \dx.\label{young}
\end{align}
Recall from \eqref{Aepsi} that 
$$A_\epsi^\sy = \nabla u_\epsi^\sy - z_\epsi^\sy - \epsi (\nabla u_\epsi z_\epsi{-} \nabla
u_\epsi w_\epsi)^\sy$$
where the remainder term $\epsi (\nabla u_\epsi z_\epsi{-} \nabla
u_\epsi w_\epsi)^\sy$ converges strongly to $0$ in $\LL^p$ for all $p
\in [1,2)$. Hence, the barycenter of $\nu^\sy$ is clearly $(\nabla
u_0^\sy {-} z^\sy_0,z_0)$.

We readily check that the measure $\nu^\sy$
is concentrated in its barycenter. Indeed, if this was not the case, by
uniform convexity we would have that
\begin{align}
   &\int_\Omega \big(|\nabla u_0^\sy {-} z_0^\sy|^2_\C + |z_0|^2_\H \big)\dx
   \nonumber\\
&< \int_\Omega \left( \int_{\Rzs\times\Rzn} \big(|A^\sy|^2_\C + |z|^2_\H\big)
    {\rm d} \nu^\sy_x(A^\sy,z)\right)\dx\nonumber
\end{align}
contradicting relation \eqref{young}. Here we have used positive
definiteness from \eqref{posdef} and \eqref{posdef2}. As $\nu^\sy$ is concentrated, we
exploit \cite[Thm. 5.4.4.iii,
p. 127]{Ambrosio-et-al08} and deduce that
$$ \int_\Omega f(x,A^\sy_\epsi(x),z_\epsi(x))\dx \to \int_\Omega
\left( \int_{\Rzs\times \Rzn} f(x,A^\sy,z) {\rm d }
  \nu^\sy_x(A^\sy,z)\right) \dx$$
along with the choice
$$ f(x,A^\sy,z) := \big|(\nabla u^\sy_0(x) {-} z_0^\sy(x),z_0^\sy(x)) - (A^\sy{-}z^\sy,z)\big|^p.$$
Hence, we have that 
$(A_\epsi^\sy,z_\epsi) \to (\nabla u_0^\sy {-} z_0^\sy,z_0)$ strongly
in $\LL^{p}(\Omega;\Rzs)\times \LL^p (\Omega;\Rzn)$ for all $p\in[1,2)$. In particular, $$\nabla u_\epsi^\sy = A_\epsi^\sy +z_\epsi^\sy +\epsi (\nabla u_\epsi z_\epsi{-} \nabla
u_\epsi w_\epsi)^\sy \to \nabla u_0^\sy \ \ \text{strongly in}  \ \
\LL^p$$ 
for all
$p\in[1,2)$ and the assertion follows by Korn's inequality.
\end{proof}

\subsection{Sketch of the proof of Theorem \ref{main2}}\label{sec:proof2}
The argument for Theorem \ref{main} can be adapted to prove
Theorem \ref{main2} as well. The only notable difference is that one has to cope with
the fact that the piecewise constant interpolants $(\ove u_\epsi, \ove
z_\epsi)$ of the approximate incremental minimizers need not be
stable but rather just {\it approximately stable}. More precisely, from \eqref{AIP} and the triangle inequality we
have that 
$$\forall (\haz u,\haz z) \in \QQ:  \EE_\epsi(t, \haz u, \haz
z) -  \EE_\epsi(t, \ove u_\epsi(t), \ove
z_\epsi(t)) + \DD_\epsi(\ove z_\epsi(t),\haz z) \geq -
\tau_\epsi\alpha_\epsi.$$
By coordinating to the sequence $(\ove u_\epsi(t),\ove z_\epsi(t))$
a mutual recovery sequence $(\haz u_\epsi, \haz z_\epsi)$ via Lemma
\ref{mutual} (with $(\ove u_\epsi(t),\ove z_\epsi(t))$
instead of $( u_\epsi(t), z_\epsi(t))$) the lower bound
\eqref{crux} still follows as $\tau_\epsi \alpha_\epsi
\to 0$. Hence, the
stability of the limit can be recovered.
Finally, improved and strong convergences in the spirit of Corollaries
\ref{ref}-\ref{ref2} can be established as well.

\subsection*{Acknowledgement}
We are gratefully indebted to the Referee for her/his careful reading of the manuscript.

\section{Appendix}
\setcounter{equation}{0}

\subsection{Estimate on left and right multiplication}
In the proof of Theorems \ref{main}-\ref{main2}
we have made use of the following estimate combining left and right
multiplication.

\begin{lemma}\label{lemmakey}
  Assume \eqref{smoothness} and \eqref{Mandel}. Then,
\begin{align}
&\exists c_7, \, c_8, \gamma >0  \ \forall G_1, \, G_2\in B_\gamma(\I) \ \forall F \in
\GLp:\nonumber\\
&|\Wel(G_1 F G_2) - \Wel(F)| \leq c_7(W(F) + c_8) \big( |G_1 {-} \I| + |G_2 {-} \I|\big).  \label{key}  
\end{align}
\end{lemma}

\begin{proof}
Following \cite[Lemma 2.5]{Ball02}, we find positive
constants $c_0, \, \haz c_0, \, \gamma$ such that, for all $G \in B_\gamma(\I)$ and
all $F \in \GLp$, one has that
\begin{align}
& \Wel(GF) \leq \haz c_0 \Wel(F) + c_0, \quad \Wel(FG) \leq \haz c_0 \Wel(F) + c_0,\label{l1}\\
& |\partial_F W(GF) F^\t| \leq \haz c_0 \Wel(F) + c_0,\label{l2}\\
& |F^\t \partial_F W(FG) | \leq \haz c_0 \Wel(F) + c_0.\label{l3}
\end{align}
For $s\in [0,1]$, let now $H_j(s):= (1{-}s)\I + s G_j$ for $j=1,2$, and note that $H_j \in
B_\gamma(\I)$. As  the derivative   $H_j' = G_j - \I$ is constant we can compute that
\begin{align}
 &\Wel(G_1 F G_2)  - \Wel(F)= \int_0^1 \frac{{\rm d}}{{\rm d} s}
 \Wel(H_1(s)F H_2(s)) \, {\rm d} s \nonumber\\
&= \int_0^1 \partial_F \Wel(H_1 F
 H_2)  {:}  (H_1' F H_2 {+} H_1 F H_2')\, {\rm d} s\nonumber\\
&=\int_0^1 \partial_F \Wel(H_1 F
 H_2) (FH_2)^\t \, {\rm d}s : H_1' + \int_0^1  (H_1 F)^\t
  \partial_F  \Wel(H_1 F
 H_2) \, {\rm d}s : H_2' .\nonumber
\end{align}
We control the above right-hand side as 
\begin{align}
 \left| \int_0^1 \partial_F \Wel(H_1 F
 H_2) (FH_2)^\t \, {\rm d}s : H_1' \right|
&\stackrel{\eqref{l2}}{\leq} \left(\int_0^1 \big(\haz c_0 \Wel(FH_2) {+} c_0\big)
\, {\rm d} s \right) |G_1 {-} \I | \nonumber\\
&\stackrel{\eqref{l1}}{\leq} \big(\haz c_0^2 \Wel(F) {+} c_0\haz c_0 {+} c_0\big)
|G_1 {-} \I |,\nonumber\\
\left|\int_0^1  (H_1 F)^\t  \partial_F  \Wel(H_1 F
 H_2) \, {\rm d}s : H_2'  \right| &\stackrel{\eqref{l3}}{\leq}\left(\int_0^1
\big(\haz c_0 \Wel(H_1 F) {+} c_0\big)
\, {\rm d} s\right) |G_2 {-} \I | \nonumber\\
&\stackrel{\eqref{l1}}{\leq} \big(\haz c_0^2 \Wel(F) {+} c_0\haz c_0{+} c_0\big)
|G_2 {-} \I |,\nonumber
\end{align}
whence the assertion follows.
\end{proof}

\subsection{Lower semicontinuity tool}

In Section \ref{proof} the following lower-semicontinuity lemma is used.

\begin{lemma}[Lower-semicontinuity]\label{balder}
Let $f_0, \, f_\epsi: \Rz^n \to [0,\infty]$ be lower semicontinuous, 
$$\forall v_0 \in \Rz^n: \quad  f_0(v_0) \leq \inf\big\{\liminf_{\epsi
  \to 0} f_\epsi(v_\epsi) \ | \
v_\epsi \to v_0\big\},$$
and $ w_\epsi  \to w_0$ weakly in $\LL^1(\Omega;\Rz^n)$. Denoting by $\nu$
the Young measure generated by $ w_\epsi  $ we have that 
$$ \int_\Omega \left( \int_{\Rz^n} f_0(w) {\rm d} \nu_x(w)\right) \dx
\leq \liminf_{\epsi \to 0} \int_\Omega
f_\epsi(w_\epsi) \dx.$$
In particular, if $f_0 $ is convex we have 
$$ \int_\Omega f_0(w_0) \dx \leq \liminf_{\epsi \to 0} \int_\Omega
f_\epsi(w_\epsi) \dx.$$
\end{lemma}

This lemma is in the same spirit of the results by {\sc Balder}
\cite[Thm. 1]{Balder84} and
{\sc Ioffe} \cite{Ioffe77} and can be proved via augmenting the variables by
including the parameter $\epsi$. The reader is referred to \cite[Thm
4.3, Cor. 4.4]{be} or
\cite[Lemma 3.1]{Mielke-et-al08} for a proof
in the case $d=1$. In case of local uniform convergence, a proof can
be found in \cite{Li96}.

\def\cprime{$'$} \def\cprime{$'$}


\end{document}